\documentclass[11pt]{article}

\usepackage{mathtools, amssymb, latexsym, amsfonts, amsthm}
\usepackage[toc,page]{appendix}
\usepackage{fancyhdr}
\usepackage{hyperref}
\mathtoolsset{showonlyrefs}

\hypersetup{
    colorlinks=true, 
    citecolor=blue, 
    linkcolor=blue, 
    urlcolor=blue, 
    linktoc=all 
}

\theoremstyle{plain}
\newtheorem{theorem}{Theorem}[section]
\newtheorem{lemma}[theorem]{Lemma}

\newtheorem{proposition}[theorem]{Proposition}
\newtheorem{fact}[theorem]{Fact}
\newtheorem{corollary}[theorem]{Corollary}
\newtheorem{conjecture}[theorem]{Conjecture}
\newtheorem{question}[theorem]{Question}

\theoremstyle{definition}
\newtheorem{definition}[theorem]{Definition}
\newtheorem{remark}[theorem]{Remark}

\newcommand{\E}{\mathop{\bf E\/}}

\newcommand\Erdos{Erd\H{o}s }
\newcommand\Rodl{R\"odl }
\newcommand\Chvatal{Chv\'atal }

\newcommand\Komlos{Koml\'os }

\newcommand{\floor}[1]{\lfloor {#1} \rfloor}
\newcommand{\ceil}[1]{\lceil {#1}\rceil}
\providecommand{\abs}[1]{\left\vert#1\right\vert}

\newcommand{\defeq}{\vcentcolon=}
\newcommand{\ind}[1]{^{(#1)}}
\newcommand{\half}{\frac{1}{2}}

\begin{document}
\title{On edge-ordered Ramsey numbers}
\author{Jacob Fox\thanks{Department of Mathematics, Stanford University, Stanford, CA 94305. Email: {\tt jacobfox@stanford.edu}. Research supported by a Packard Fellowship, and by NSF Career Award DMS-1352121.}\,\, and Ray Li\thanks{Department of Computer Science, Stanford University, Stanford, CA 94305. Email: {\tt rayyli@cs.stanford.edu}.  Research supported by the National Science Foundation Graduate Research Fellowship Program under Grant No. DGE-1656518.}}
\date{\today}
\maketitle
\newcommand{\edge}{\text{edge}}
\newcommand{\edgeind}{\text{edge, ind}}
\newcommand\inuce{\text{ind}}
\begin{abstract}
An edge-ordered graph is a graph with a linear ordering of its edges. 
Two edge-ordered graphs are \emph{equivalent} if their is an isomorphism between them preserving the ordering of the edges.
The \emph{edge-ordered Ramsey number} $r_{\edge}(H; q)$ of an edge-ordered graph $H$ is the smallest $N$ such that there exists an edge-ordered graph $G$ on $N$ vertices such that, for every $q$-coloring of the edges of $G$, there is a monochromatic subgraph of $G$ equivalent to $H$.
Recently, Balko and Vizer announced that $r_{\edge}(H;q)$ exists. 
However, their proof uses the Graham-Rothschild theorem and consequently gives an enormous upper bound on these numbers. 
We give a new proof giving a much better bound. 
We prove that for every edge-ordered graph $H$ on $n$ vertices, we have $r_{\edge}(H;q) \leq 2^{c^qn^{2q-2}\log^q n}$, where $c$ is an absolute constant. 
We also explore the edge-ordered Ramsey number of sparser graphs and prove a polynomial bound for edge-ordered graphs of bounded degeneracy. We also prove a strengthening for edge-labeled graphs, graphs where every edge is given a label and the labels do not necessary have an ordering.
\end{abstract}

\newcommand\toedge{\xrightarrow{edge}}
\section{Introduction}
The \emph{Ramsey number}, $r(H)$, of a graph $H$ is the smallest $N$ such that, in any two-coloring of the edges of a complete graph on $N$ vertices, there exists a monochromatic copy of $H$. 
The existence of these numbers was first proved by Ramsey \cite{Ramsey30}.

This work considers Ramsey numbers for \emph{edge-ordered} graphs.
Given a graph $G=(V,E)$ with vertex set $V=[N]$, an \emph{edge-ordering} of $G$ is a total ordering of the edges of $G$.
Alternatively, we may define an edge-ordering of $G$ by assigning a unique integer label to each edge in the graph. 
Two edge-ordered graphs $H_1$ and $H_2$ are \emph{equivalent} if their is an isomorphism between $H_1$ and $H_2$ preserving the ordering of the edges.
We write $G\xrightarrow[q]{\edge}H$ if any $q$-coloring of the edges of $G$ contains a monochromatic copy of $H$, i.e.\ a monochromatic edge-ordered subgraph equivalent to $H$.
If $H$ is an edge-ordered graph, the \emph{edge-ordered Ramsey number} $r_{\edge}(H; q)$ is the smallest $N$ such that there exists an edge-ordered graph $G$ on $N$ vertices such that $G\xrightarrow[q]{\edge}H$.
Here, $q$ is understood to be 2 if omitted.
In this work, we consider the question of bounding $r_{\edge}(H; q)$.

Clearly, $r(H)\le r_{\edge}(H)$ for every edge-ordered graph $H$, where $r(H)$ is the Ramsey number of the underlying graph (without an edge-ordering).
In the special case where $H$ is a complete graph on $n$ vertices with the edges lexicographically ordered,\footnote{An edge-ordering is \emph{lexicographical} if one can number the vertices of $H_2$ by $1,\dots,n$ such that the edges $(i,i')$ of $H_2$ with $1\le i<i'\le n$ are ordered according to the lexicographical order of the pair $(i,i')$} by choosing a lexicographical ordering of $G$, the edge-ordered Ramsey number $r_{\edge}(H)$ is simply the usual Ramsey number $r(K_n)$.
However, it is not obvious that the edge-ordered Ramsey number exists in general.
Balko and Vizer \cite{BalkoVslides} proved that these numbers in fact do exist.
\begin{theorem}[\cite{BalkoVslides}]\label{theorem1}
  For each edge-ordered graph $H$, there exists an edge-ordered graph $G$ such that any two-coloring of the edges of $G$ contains a monochromatic copy of $H$.
\end{theorem}
However, their proof is based on the Graham-Rothschild theorem \cite{BalkoVslides}, which gives an enormous bound on the edge-ordered Ramsey number \cite{GrahamR71,Shelah88}. Theorem \ref{theorem1} also follows from a general Ramsey-type theorem of Hubi\v cka and Ne\v set\v ril, Theorem 4.33 in \cite{HuNe}.
Their proof is constructive but also gives an enormous bound.
We improve the bound on these Ramsey numbers to single exponential type.
\begin{theorem}
\label{thm:edge-1}
For each positive integer $n$, there is an edge-ordered graph $G$ on $N = \exp(100n^2\log^2n)$ vertices\footnote{All logs and exps are base 2 unless otherwise specified.} such that, for every two-coloring of the edges of $G$, there exists a monochromatic subgraph containing a copy of every $n$-vertex edge-ordered graph $H$.
  \end{theorem}
Theorem~\ref{thm:edge-1} gives the following immediate corollary.
\begin{corollary}
If $H$ is an edge-ordered graph on $n$ vertices, then
\begin{align}
  r_{\edge}(H)\le 2^{100n^2\log^2n}.
\end{align}
\end{corollary}
For more than two colors, one can deduce a bound from the two color case.
Note that if $F\xrightarrow[s]{edge}G$ and $G\xrightarrow[r]{edge}H$, then $F\xrightarrow[rs]{edge}H$. Indeed, 
given a $rs$-coloring of $F$, we can partition the set of $rs$ colors into $s$ subsets of $r$ colors, find a copy of $G$ in $F$ with using only colors from one such subset, and find a monochromatic copy of $H$ within this copy of $G$.
Applying this recursively with $r=2$ gives a bound on $r_\edge(H; q)$ which is an exponential tower of $n$'s of height roughly $\log q$.
However, carefully extending the proof of Theorem~\ref{thm:edge-1} gives a much better bound.
\begin{theorem}
\label{thm:multi}
  Let $n$ and $q$ be positive integers with $q\ge 2$.
  There exists an edge-ordered graph $G$ on $N\defeq\exp(8^{q+1}n^{2q-2}(\log n)^q)$ vertices such that, for every $q$-coloring of the edges of $G$, there exists a monochromatic subgraph containing a copy of every $n$-vertex edge-ordered graph $H$.
\end{theorem}
Theorem~\ref{thm:multi} gives the following immediate corollary.
\begin{corollary}
For every edge-ordered graph $H$ on $n$ vertices and integer $q\ge 2$, we have
\begin{align}
  r_{\edge}(H; q)\le 2^{8^{q+1}n^{2q-2}(\log n)^q}.
\end{align}
\end{corollary}

After the Ramsey number of the complete graph, the Ramsey number of sparse graphs is one of the most studied topics in Ramsey theory.
A natural notion of sparseness in a graph is degeneracy.
A graph is \emph{$d$-degenerate} if every induced subgraph has a vertex of degree at most $d$.
Burr and \Erdos \cite{BE75} conjectured that, for every positive integer $d$, there exists a constant $c(d)$ such that every $d$-degenerate graph $H$ on $n$ vertices has Ramsey number at most $c(d)n$.
This conjecture was proved recently by Lee~\cite{Lee17}. 

It is natural to study analogous questions for edge-ordered Ramsey numbers.
Balko and Vizer \cite{BalkoVslides} proved some results in this direction with additional assumptions about the edge-ordered graphs.
We prove a polynomial bound for edge-ordered Ramsey number of graphs of bounded degeneracy.
\begin{theorem}
  \label{thm:deg}
  If $H$ is an edge-ordered $d$-degenerate graph on $n$ vertices, then 
  \begin{align}
  r_{\edge}(H)\le n^{600d^3\log(d+1)}.
  \end{align}
\end{theorem}

In the above, we defined an edge-ordering of a graph by assigning a \emph{unique} integer label to each edge of the graph.
We can generalize the notion of edge-orderings to \emph{edge-orderings with possibly repeated labels}, where the integer labels are not necessarily unique.
In this case, we say two edge-ordered graphs with possibly repeated labels $H_1$ and $H_2$ are \emph{equivalent} if there is an isomorphism between them preserving the ordering of the integer labels.
In particular, equally labeled edges must be mapped to equally labeled edges.
We again write $G\xrightarrow{\edge}H$, if every two coloring of the edges of $G$ contains a monochromatic copy of $H$.
We show that, in this case, edge-ordered Ramsey numbers still exist.
In fact, we obtain a stronger result for \emph{edge-labeled graphs}, graphs where every edge is given a label and the labels do not necessary have an ordering.
For edge-labeled graphs $G$ and $H$, we say $G$ contains a copy of $H$ if there is a subgraph of $G$ with an isomorphism to $H$ preserving the edge labels.
\begin{theorem}
\label{thm:rpt}
Let $q,n,m \ge 2$, and $M=q(m-1)+1$. There exists an edge-labeled graph $G$ with labels in $[M]$ such that, for every $q$-coloring of the edges of $G$, there is a set $S$ of $m$ distinct labels and a color $c$ such that every possible edge-labeling of a clique on $n$ vertices with labels in $S$ appear monochromatically in color $c$. In particular, if we let $m={n \choose 2}$, the edge-labeled graph $G$ contains a monochromatic in color $c$ copy of every edge-ordered graph with possibly repeated labels on $n$ vertices. 
\end{theorem}

Our argument uses a multicolor variant of the cylinder regularity lemma of Duke, Lefmann, and \Rodl \cite{DLR95} and chooses a graph $G$ whose size is triple exponential in $m$ and double-exponential in $q$ and $n$. One could alternatively prove this result by further developing the techniques already used in this work and it gives a graph $G$ whose size is of double exponential type in $m,q,n$. However, for brevity, we chose to only include the proof of Theorem~\ref{thm:rpt} using the cylinder regularity lemma.

Beyond complete graphs and sparse graphs, there is a rich study of Ramsey numbers of other graphs (see for example, the survey \cite{ConlonFS15}). It is natural to consider analogous questions for edge-ordered graphs, and we do so in Section~\ref{sec:conclusion}.

\vspace{1em}
\noindent\textbf{Organization of the paper.}
In Section~\ref{sec:prelim} we establish some notation and results that we use in the proofs of our main theorems.
In Section~\ref{sec:outline}, we outline the proofs of Theorems~\ref{thm:edge-1} and \ref{thm:multi}.
In Section~\ref{sec:two}, we prove Theorem~\ref{thm:edge-1}.
In Section~\ref{sec:multi}, we prove Theorem~\ref{thm:multi}.
In Section~\ref{sec:sparse}, we prove Theorem~\ref{thm:deg}.
In Section~\ref{sec:rpt}, we prove Theorem~\ref{thm:rpt}.
In Section~\ref{sec:conclusion}, we discuss a variety of results and open problems on edge-ordered Ramsey numbers.

\section{Preliminaries}
\label{sec:prelim}

\subsection{Notation}

Throughout the paper we omit floors and ceilings when they are not crucial.
All log's and exp's are base 2 unless otherwise specified.
For a positive integer $n$, let $[n]$ denote the set $\{1,\dots,n\}$.
An \emph{interval} is a set of consecutive integers.
For positive reals $a_1,a_2,a_3$, let $(a_1\pm a_2)a_3$ denote the range $[(a_1-a_2)a_3, (a_1+a_2)a_3]$.
A partition $X_1\cup\cdots\cup X_n$ of a set $X$ is \emph{equitable} if the size of any two $X_i$'s differs by at most 1. 
In this work, the letter $F$ typically denotes a graph with no edge-ordering and the letters $G$ and $H$ typically denote edge-ordered graphs.
In a graph $H$ (edge-ordered or not), let $V(H)$ and $E(H)$ denote the set of vertices and edges of $H$, respectively.
We use the following notation for unordered graphs.
\begin{definition}
  Let $F$ be a graph.
  For vertices $v$ and $w$, vertex subset $X$, and disjoint vertex subsets $V$ and $W$, let
  \begin{align}
    \deg\ind{F}(v) \ &\defeq \ \abs{\{u: uv\in E(F)\}} \nonumber\\
    \deg\ind{F}(v,X) \ &\defeq \  \abs {\{u\in X: uv\in E(F)\}} \nonumber\\
    \deg\ind{F}(v,w,X) \ &\defeq \  \abs {\{u\in X: uv, uw\in E(F)\}} \nonumber\\
    e\ind{F}(V,W) \ &\defeq \ \abs{\{(v,w)\in V\times W: vw\in E(F)\}}  \nonumber\\
    d\ind{F}(V,W) \ &\defeq \ \frac{e\ind{F}(V,W)}{|V||W|}.
  \label{}
  \end{align}
\end{definition}
We use the following notation for edge-ordered graphs.
\begin{definition}
  Let $G$ be an edge-ordered graph.
  For an interval $I\subset[\binom{N}{2}]$, let $G_I$ denote the unordered graph on vertex set $[N]$ whose edges are the edges of $G$ with label in $I$.
  For each $I$, let $\alpha_I \defeq  \frac{|I|}{\binom{N}{2}}$.
  Let $\deg_I\ind{G}, e_I\ind{G},$ and $d_I\ind{G}$ be alternative notations for $\deg\ind{G_I},e\ind{G_I}$, and $d\ind{G_I}$, respectively, and for simplicity we often write $\deg_I, e_I$, and $d_I$ when the edge-ordered graph $G$ is understood from context.
\end{definition}
A graph is \emph{$d$-degenerate} if every induced subgraph has a vertex of degree at most $d$.
Equivalently, a graph is $d$-degenerate if and only if there exists an ordering of the vertices $1,\dots,n$ such that, for all $i$, the number of neighbors $j<i$ of $i$ is at most $d$.
This ordering shows that all $d$-degenerate graphs are $d+1$-colorable.

\subsection{Regularity in edge-ordered graphs}

We use the following notion of regularity.
\begin{definition}
  In a graph $F$, a pair of disjoint vertex subsets $(X, Y)$ is \emph{$(\alpha, \varepsilon)$-regular} if, for all $X'\subset X$ and $Y'\subset Y$ with $|X'|\ge \varepsilon |X|$ and $|Y'|\ge \varepsilon|Y|$, we have
  \begin{align}
    \label{eq:reg-1}
    \abs{d\ind{F}(X',Y') - \alpha} < \varepsilon.
  \end{align}
\end{definition}
The next fact follows immediately from the definition.
\begin{fact}
\label{fact:reg-1}
Let $\alpha,\alpha',\varepsilon,\varepsilon'>0$ satisfy $\varepsilon \ge \varepsilon' + |\alpha-\alpha'|$.
Every $(\alpha',\varepsilon')$-regular pair is also $(\alpha,\varepsilon)$-regular.
\end{fact}
The following standard lemma shows that most vertex degrees between two subsets in a regular pair are near the average degree.
\begin{lemma}
  \label{lem:reg-1}
  If $F$ is a graph and the pair of vertex subsets $(X,Y)$ is $(\alpha,\varepsilon)$-regular, there are at most $2\varepsilon |X|$ vertices $x$ in $X$ such that $\deg(x, Y)\notin (\alpha\pm \varepsilon)|Y|$.
\end{lemma}
\begin{proof}
  Suppose for sake of contradiction that there are more than $2\varepsilon |X|$ vertices $x\in X$ such that $\deg(x,Y)\notin (\alpha\pm \varepsilon)|Y|$.
  Then there exists a set $X'$ of at least $\varepsilon |X|$ vertices such that either $\deg(x',Y)<(\alpha-\varepsilon)|Y|$ for all $x'\in X'$ or $\deg(x',Y)>(\alpha+\varepsilon)|Y|$ for all $x'\in X'$.
  In either case, $|d\ind{F}(X',Y)-\alpha|>\varepsilon$, implying that the pair $(X,Y)$ is not $(\alpha,\varepsilon)$-regular, a contradiction.
\end{proof}
The next standard lemma shows that regularity is inherited in large subsets.
\begin{lemma}
  Suppose that, in a graph $F$, vertex subsets $X,X',Y,Y'$ are such that the pair $(X,Y)$ is $(\alpha,\varepsilon)$-regular, $X'\subset X$, and $Y'\subset Y$.
  Then, for $\varepsilon'=\varepsilon\cdot\max(|X|/|X'|, |Y|/|Y'|)$, the pair $(X',Y')$ is $(\alpha,\varepsilon')$-regular.
\label{lem:reg-2}
\end{lemma}
\begin{proof}
  For all $X''\subset X'$ and $Y''\subset Y'$ with $|X''|\ge \varepsilon'|X'|$, and $|Y''|\ge \varepsilon'|Y'|$, we have $|X''|\ge \varepsilon|X|$ and $|Y''|\ge \varepsilon|Y|$, so by $(\alpha,\varepsilon)$-regularity of the pair $(X,Y)$, we have
  \begin{align}
    |d\ind{F}(X'',Y'')-\alpha| < \varepsilon \le \varepsilon'.
  \label{}
  \end{align}
  This proves that the pair $(X',Y')$ is $(\alpha,\varepsilon')$-regular.
\end{proof}
The next lemma shows that, across regular pairs, most co-degrees are near what one would expect in a random graph with the same edge density.
\begin{lemma}
\label{lem:reg-3}
  Let $\alpha$ and $\varepsilon$ be such that $\alpha\in(0,\frac{1}{2})$ and $\varepsilon\in(0,\frac{\alpha}{5})$.
  Let $F$ be a graph and $X,Y,Z$ be pairwise disjoint subsets of vertices such that the pairs $(X,Z)$ and $(Y,Z)$ are $(\alpha,\varepsilon)$-regular.
  Then,
  \begin{align}
    \#\left\{(x,y)\in X\times Y: \deg\ind{F}(x,y,Z) \notin (\alpha^2\pm 2\varepsilon)|Z| \right\}
    \ \le \ 4\varepsilon\alpha^{-1}|X||Y|.
  \label{}
  \end{align}
\end{lemma}
\begin{proof}
  Call a vertex pair $(x,y)\in X\times Y$ \emph{good} if 
  \begin{align}
    \deg\ind{F}(x,y,Z) \in (\alpha^2\pm 2\varepsilon)|Z|
  \end{align}
  and \emph{bad} otherwise.
  By Lemma~\ref{lem:reg-1} on the $(\alpha,\varepsilon)$-regular pair $(X,Z)$, there exists a subset $X'$ of $X$ such that $|X\setminus X'|\le 2\varepsilon |X|$ and such that, for all $x\in X'$, we have $\deg\ind{F}(x, Z)\in (\alpha\pm \varepsilon)|Z|$.

  Fix $x\in X'$.
  Let $Z_{x}$ denote the neighbors of $x$ in $Z$.
  We have
  \begin{align}
    |Z_{x}|=\deg\ind{F}(x, Z) \ge (\alpha-\varepsilon)|Z|.
  \end{align}
  Hence, as the pair $(Y,Z)$ is $(\alpha,\varepsilon)$-regular, by Lemma~\ref{lem:reg-2}, the pair $(Y,Z_x)$ is $(\alpha,\frac{\varepsilon}{\alpha-\varepsilon})$-regular.
  By Lemma~\ref{lem:reg-1} on the pair $(Y,Z_x)$, there exists a set $Y_x$ consisting of all but at most $2\frac{\varepsilon}{\alpha-\varepsilon}|Y|$ vertices of $Y$ such that, for all $y\in Y_x$, we have $\deg\ind{F}(y, Z_{x}) \in (\alpha\pm\varepsilon)|Z_{x}|$.
  In this case, all such $y$ satisfy
  \begin{align}
    \deg\ind{F}(x,y,Z)
    \ = \  \deg\ind{F}(y,Z_{x})
    \ \le \ \left(\alpha + \frac{\varepsilon}{\alpha - \varepsilon}\right)|Z_{x}|
    \ < \  (\alpha^2 + 2\varepsilon)|Z|,
  \label{}
  \end{align}
  where in the last inequality we used $|Z_x|\le (\alpha+\varepsilon)|Z|$ and the bound $(\alpha+\frac{\varepsilon}{\alpha-\varepsilon})(\alpha+\varepsilon) =  \alpha^2+\frac{\alpha\varepsilon}{\alpha-\varepsilon} + \alpha\varepsilon + \frac{\varepsilon^2}{\alpha-\varepsilon} < \alpha^2 + \frac{5}{4}\varepsilon + \frac{1}{2}\varepsilon + \frac{1}{4}\varepsilon = \alpha^2+2\varepsilon$.
  Similarly, $\deg\ind{F}(x,y,Z) > (\alpha^2 - 2\varepsilon)|Z|$. 
  Hence, the vertex pair $(x,y)$ is good.
  Thus, there are at most $2\frac{\varepsilon}{\alpha-\varepsilon}|Y|$ bad vertex pairs for each of the at most $|X|$ vertices $x\in X'$.
  There are also at most $|Y|$ bad vertex pairs for each of the at most $2\varepsilon |X|$ vertices $x\in X\setminus X'$. 
  This gives a total of at most $(\frac{2\varepsilon}{\alpha-\varepsilon} + 2\varepsilon)|X||Y| < 4\varepsilon\alpha^{-1}|X||Y|$ bad vertex pairs.
\end{proof}

We now define $\varepsilon$-regularity for edge-ordered complete graphs and prove the existence of $\varepsilon$-regular edge-ordered complete graphs.
\begin{definition}
  An edge-ordered complete graph $G$ on $N$ vertices is \emph{$\varepsilon$-regular} if, for all intervals $I\subset\binom{N}{2}$ of length at least $\varepsilon\binom{N}{2}$, and all subsets $X$ and $Y$ of size at least $\varepsilon N$, the pair $(X,Y)$ is $(\alpha_I, \varepsilon)$-regular in the graph $G_I$.
\end{definition}

\begin{lemma}
\label{lem:edge-0}
  For all $\varepsilon\in(0,1/2)$ and for all $N\ge (2/\varepsilon)^6$, there exists an $\varepsilon$-regular edge-ordered complete graph on $N$ vertices.
\end{lemma}
\begin{proof}
  Let $N\ge (2/\varepsilon)^6$.
  Let $G$ be a complete graph on $N$ vertices, and define an edge-ordering of $G$ by labeling the edges from a uniformly random permutation of $1,\dots,\binom{N}{2}$.
  To prove that $G$ is $\varepsilon$-regular, it suffices to prove that, for every interval $I$ of length at least $\varepsilon\binom{N}{2}$ and every pair $(X',Y')$ of disjoint subsets each of size at least $\varepsilon^2 N$, we have
  \begin{align}
    |d_I(X',Y')-\alpha_I| < \varepsilon.
  \label{eq:reg-5}
  \end{align}
  This guarantees that any pair $(X,Y)$ of disjoint sets each of size at least $\varepsilon N$ is $(\alpha_I,\varepsilon)$-regular.

  Call a triple $(I, X',Y')$ \emph{bad} if \eqref{eq:reg-5} fails, $|I|\ge \varepsilon \binom{N}{2}$, and $|X'|,|Y'|\ge \varepsilon^2 N$.
  We claim that, with high probability, no triple $(I, X', Y')$ is bad.
  Fix an interval $I$ of size at least $\varepsilon\binom{N}{2}$ and vertex subsets $X'$ and $Y'$ of size at least $\varepsilon^2 N$.
  The unordered graph $G_I$ has precisely the same distribution as $G_{N,|I|}$, a uniformly random graph with exactly $|I|$ edges.
  Hence, the number of edges between $X'$ and $Y'$ with label in $I$ is distributed as a $(\binom{N}{2},|I|,|X'||Y'|)$ hypergeometric distribution, which is at least as concentrated as the corresponding binomial distribution (see, for example, Section 6 of \cite{Hoeffding63}), which has mean $\alpha_I|X'||Y'|$.
  Thus, we may apply the Chernoff bound to obtain
  \begin{align}
    \Pr_G\left[ (I, X',Y')\text{ bad} \right]
    \ &= \ \Pr_G\left[ \abs{e_I\ind{G}(X',Y')-\alpha_I |X'||Y'|} > \frac{\varepsilon}{\alpha_I}\cdot \alpha_I |X'||Y'| \right]  \nonumber\\
    \ &\le \   \exp\left(-\frac{1}{2}\cdot \frac{\varepsilon^2}{\alpha_I^2}\cdot \alpha_I |X'||Y'|\right)\nonumber\\
    \ &\le \ \exp\left( -\frac{1}{2} \varepsilon^6N^2 \right)  \nonumber\\
    \ &\le \ \exp(-32N).
  \label{}
  \end{align}
  The second inequality used that $|X'|,|Y'|\ge \varepsilon^2N$ and $\alpha_I\le 1$. 
  The last inequality used that $\varepsilon\ge 2N^{-1/6}$.
 
  There are at most $N^4$ intervals $I$ of length at least $\varepsilon\binom{N}{2}$, because there are at most $N^2$ choices for each endpoint.
  There are at most $2^{2N}$ pairs of disjoint subsets vertices $(X',Y')$ each of size at least $\varepsilon^2 N$.
  Hence, by the union bound, we have,
  \begin{align}
    \Pr\left[\text{$G$ is not $\varepsilon$-regular} \right]
    \ &= \ \Pr\left[\text{Exists $I, X',Y'$ such that $(I, X',Y')$ is bad} \right]  \nonumber\\
    \ &< \ N^4\cdot 2^{2N}\cdot \exp(-32N)
    \ <  \ 1.
  \label{}
  \end{align}
  Thus, there exists an edge-ordered complete graph $G$ on $N$ vertices that is $\varepsilon$-regular.
\end{proof}
\begin{remark}
  If there exist $X'$ and $Y'$ of size at least $\varepsilon^2N$, violating \eqref{eq:reg-5}, then, by considering random subsets, there also exist subsets of size exactly $\varepsilon^2N$ violating $\eqref{eq:reg-5}$.
  Hence, instead of considering subsets $X'$ and $Y'$ of size at least $\varepsilon^2N$, it suffices to union bound over subsets $X'$ and $Y'$ of size exactly $\varepsilon^2N$, of which there are $\binom{N}{\varepsilon^2 N}\le \exp(\varepsilon^{2-o(1)}N)$ choices for each of $X$ and $Y$.
  In doing so, we can improve the bound $N\ge(2/\varepsilon)^6$ to $N\ge \varepsilon^{4+o(1)}$, but this does not significantly improve our main results.
\end{remark}

\subsection{Counting lemma}

The following lemma counts the number of $n$-cliques spanning $n$ pairwise regular vertex subsets. 
\begin{lemma}
  \label{lem:count}
  Let $p\in (0,\frac{1}{2})$ and $p_{i,j}\in(p,1)$ for $1\le i<j\le n$.
  Let $\varepsilon$ satisfy $0<\varepsilon\le (p/4)^n$.
  Let $F$ be a graph and $W_1,\dots,W_n$ be pairwise disjoint vertex subsets such that the pair $(W_i,W_j)$ is $(p_{i,j},\varepsilon)$-regular for $1\le i < j\le n$.
  The number of $n$-cliques with one vertex in each $W_i$ is in the range
  \begin{align}
    \left(1\pm \frac{4\varepsilon n}{p^n} \right)\cdot \prod_{i=1}^{n} |W_i|\cdot \prod_{1\le i < j\le n} p_{i,j}.
  \label{}
  \end{align}
\end{lemma}
\begin{proof}
  We induct on $n$.
  The base case $n=1$ is true, as there are exactly $|W_1|$ cliques of size 1.

  Now assume $n\ge 2$ and the assertion is true for $n-1$.
  Call a vertex $w_n\in W_n$ \emph{good} if, 
  \begin{align}
    \label{eq:edge-7}
    \deg\ind{F}(w_n, W_i)\in (p_{i,n}\pm \varepsilon)|W_i|
  \end{align}
  for all $i<n$ and \emph{bad} otherwise.
  For each $i<n$, by Lemma~\ref{lem:reg-1} on the $(p_{i,n},\varepsilon)$-regular pair $(W_i,W_n)$, all but at most $2\varepsilon|W_n|$ vertices $w_n\in W_n$ satisfy \eqref{eq:edge-7}.
  By taking a union over all $i=1,\dots,n-1$, all but at most $2\varepsilon n|W_n|$ vertices $w_n\in W_n$ are good.

  We count the number of $n$-cliques $w_1,\dots,w_n$ with $w_i\in W_i$ for all $i$ by caseworking on $w_n$.
  Let $M$ be the number of such $n$-cliques.
  Let $M_{good}$ be the number of such $n$-cliques where vertex $w_n$ is good.
  Similarly, let $M_{bad}$ be the number of such $n$-cliques where vertex $w_n$ is bad.
  In this way, $M=M_{good}+M_{bad}$.

  First suppose $w_n$ is good. 
  For $i=1,\dots, n-1$, let $W_i'$ be the neighbors of $w_n$ in $W_i$.
  As $w_n$ is good, we have that, for all $i=1,\dots,n-1$,
  \begin{align}
    \abs{W_i'} \ &= \  (p_{i,n}\pm \varepsilon)|W_i|.
  \end{align}
  As $\max(\frac{|W_i|}{|W_i'|},\frac{|W_j|}{|W_j'|})\le \frac{1}{p-\varepsilon}$, by Lemma~\ref{lem:reg-2}, the pair $(W_i',W_j')$ is $(p_{i,j}, \frac{\varepsilon}{p-\varepsilon})$-regular for all $1\le i<j\le n-1$.

  The number of $n$-cliques containing $w_n$ is exactly the number of $(n-1)$-cliques spanning $W_1',\dots,W_{n-1}'$.
  As $(p/4)^{n-1}\ge \frac{\varepsilon}{p/4} > \frac{\varepsilon}{p-\varepsilon}$, we can apply the induction hypothesis on $W_1',\dots,W_{n-1}'$ with $p_{i,j}'=p_{i,j}$ for $1\le i < j \le n-1$ and $\varepsilon'=\frac{\varepsilon}{p-\varepsilon}$ to obtain that the number of $(n-1)$-cliques spanning $W_1',\dots,W_{n-1}'$  is at most
  \begin{align}
    \ & \  \left( 1 + \frac{4\frac{\varepsilon}{p-\varepsilon}(n-1)}{p^{n-1}} \right)
      \cdot  \prod_{i=1}^{n-1} |W_i'|\cdot \prod_{1\le i<j\le n-1}^{} p_{i,j} \nonumber\\
    \ &\le \ \left( 1+ \frac{4\varepsilon n - 3\varepsilon}{p^n} \right) 
      \cdot \prod_{i=1}^{n-1} |W_i'| \cdot \prod_{1\le i<j\le n-1}^{} p_{i,j} \nonumber\\
    \ &\le \  \left( 1+ \frac{4\varepsilon n - 3\varepsilon}{p^n} \right) 
      \cdot \left(1 + \frac{\varepsilon}{p}\right)^{n-1} \cdot\prod_{i=1}^{n-1} |W_i|\cdot \prod_{1\le i<j\le n}^{} p_{i,j} \nonumber\\
    \ &\le \  \left( 1+ \frac{4\varepsilon n-3\varepsilon}{p^n} + \frac{2(n-1)\varepsilon}{p}\right) 
      \cdot \prod_{i=1}^{n-1} |W_i|\cdot \prod_{1\le i<j\le n}^{} p_{i,j} \nonumber\\
    \ &\le \  \left( 1+ \frac{4\varepsilon n -2\varepsilon}{p^n} \right) 
      \cdot \prod_{i=1}^{n-1} |W_i|\cdot \prod_{1\le i<j\le n}^{} p_{i,j}.
  \label{eq:edge-9}
  \end{align}
  In the first inequality, we used $\varepsilon\le(\frac{p}{4})^n<\frac{p}{4n-3}$ so $\frac{4(n-1)}{p-\varepsilon} < \frac{4n-3}{p}$.
  In the second inequality, we used that $|W_i'|\le |W_i|(p_{i,n}+\varepsilon) \le (1+\frac{\varepsilon}{p})|W_i|p_{i,n}$ for all $i=1,\dots,n-1$.
  In the third inequality, we used that $(1+y)(1+x)\le 1+y+2x$ for $x,y\in(0,1)$ and that $\frac{4\varepsilon n - 3\varepsilon}{p^n}$ and $\frac{\varepsilon(n-1)}{p}$ are less than $1$ for all $n,\varepsilon,p$ given by the lemma's assumptions.
  In the last inequality, we used that $2(n-1)\le\frac{1}{p^{n-1}}$.
  There are at most $|W_n|$ good vertices $w_n$, so the total number $M_{good}$ of $n$-cliques containing a good vertex in $W_n$ is at most $|W_n|$ times the term in the last line of \eqref{eq:edge-9}.
  Hence,
  \begin{align}
    M_{good} 
    \ &\le \  \left( 1+ \frac{4\varepsilon n -2\varepsilon}{p^n} \right) \cdot \prod_{i=1}^{n} |W_i|\cdot \prod_{1\le i<j\le n}^{} p_{i,j}.
  \label{}
  \end{align}

  Now suppose that $w_n$ is bad.
  There are at most $2\varepsilon n|W_n|$ choices of $w_n$.
  The number of $n$-cliques containing $w_n$ is bounded above by the number of $(n-1)$-cliques spanning $W_1,\dots,W_{n-1}$.
  Applying the induction hypothesis on sets $W_1,\dots,W_{n-1}$ with the same values of $p_{i,j}$ for $1\le i<j\le n-1$ and the same $\varepsilon$, we have that each such $w_n$ is part of at most 
  \begin{align}
    \left( 1 + \frac{4\varepsilon(n-1)}{p^{n-1}} \right)\cdot \prod_{i=1}^{n-1} |W_i| \cdot \prod_{1\le i<j\le n-1}^{} p_{i,j}
    \ \le \  
    2\cdot \prod_{i=1}^{n-1} |W_i| \cdot \prod_{1\le i<j\le n-1}^{} p_{i,j}
    \label{eq:edge-10}
  \end{align}
  $(n-1)$-cliques.
  Hence, the total number of $n$-cliques containing a bad $w_n$ satisfies
  \begin{align}
    \label{eq:edge-11}
    M_{bad}
    \ \le \  2\varepsilon|W_n|\cdot 2\cdot \prod_{i=1}^{n-1} |W_i| \cdot \prod_{1\le i<j\le n-1}^{} p_{i,j}
    \ \le \  \frac{4\varepsilon}{p^{n-1}} \cdot \prod_{i=1}^{n} |W_i| \cdot \prod_{1\le i<j\le n}^{} p_{i,j}.
  \end{align}
  It follows that the total number of $n$-cliques, $M$, satisfies
  \begin{align}
    M \ &= \ M_{good} + M_{bad} \nonumber\\
    \ &\le \   \left( 1 + \frac{4\varepsilon n - 2\varepsilon}{p^n} + \frac{4\varepsilon}{p^{n-1}} \right) \cdot \prod_{i=1}^{n} |W_i|\cdot \prod_{1\le i<j\le n}^{} p_{i,j}  \nonumber\\
    \ &\le \   \left( 1+ \frac{4\varepsilon n}{p^n} \right) \cdot \prod_{i=1}^{n} |W_i|\cdot \prod_{1\le i<j\le n}^{} p_{i,j}. 
  \label{}
  \end{align}

  By a similar computation to \eqref{eq:edge-9}, and using the inequality the inequality $(1-a)(1-b) > 1-a-b$ for positive $a$ and $b$, we obtain that the number of $n$-cliques in $F$ containing a particular good $w_n$ is at least
  \begin{align}
    \left( 1- \frac{4\varepsilon n - 2\varepsilon}{p^n} \right) \cdot \prod_{i=1}^{n-1} |W_i|\cdot \prod_{1\le i<j\le n}^{} p_{i,j}.
  \label{}
  \end{align}
  There are at least $(1-2\varepsilon n)|W_n|$ good vertices of $w_n$ by an earlier argument, so the number of $n$-cliques containing a good vertex is at least $(1-2\varepsilon n)|W_n|$ times the above.
  Hence,
  \begin{align}
    M
    \ &\ge \ M_{good} \nonumber\\ 
    \ &> \  \left( 1- \frac{4\varepsilon n - 2\varepsilon}{p^n} - 2\varepsilon n\right) \cdot \prod_{i=1}^{n} |W_i|\cdot \prod_{1\le i<j\le n}^{} p_{i,j} \nonumber\\
    \ &\ge \  \left( 1- \frac{4\varepsilon n}{p^n} \right) \cdot \prod_{i=1}^{n} |W_i|\cdot \prod_{1\le i<j\le n}^{} p_{i,j},
  \label{}
  \end{align}
  as desired.
\end{proof}

\section{Proof outline of Theorems~\ref{thm:edge-1} and \ref{thm:multi}}
\label{sec:outline}

Here, we outline the proof of Theorem~\ref{thm:edge-1}.
At the end of the section, we also describe the ideas needed to extend Theorem~\ref{thm:edge-1} to the multicolor result of Theorem~\ref{thm:multi}.

Let $N = 2^{100n^2\log^2n}$, $\varepsilon = 2^{-16n^2\log^2n}$, and $\delta_2=2^{-6n\log n}$.
Let $G$ be an $\varepsilon$-regular edge-ordered complete graph, which exists by Lemma~\ref{lem:edge-0}.

Consider any red/blue coloring of the edges of $G$ with no red copy of some edge-ordered complete graph $H$ on $n$ vertices.
We show $G$ has a blue copy of $H$.\footnote{The same technique shows $G$ has a blue copy of every other edge-ordered $H'$ on $n$ vertices.}
Let $G_1$ be the edge-ordered subgraph of $G$ consisting of the red edges of $G$ and their corresponding edge labels. 
We prove a technical lemma, Lemma~\ref{lem:edge-1}, which says that if $G_1$ has no copy of $H$, then $G_1$ satisfies a sparseness property. 
Formally, an edge-ordered graph $G_1$ is $(\alpha,\gamma,\delta,t)$-sparse if, for any interval $I$ with $|I|\ge\alpha^{-1}$ and any vertex subset $X$ with $|X|\ge \gamma^{-1}$, there exists an interval $I'\subset I$ of size at least $\alpha|I|$ and pairwise disjoint vertex subsets $W_1,\dots,W_t$ each of size at least $\gamma|X|$ such that $d_I\ind{G_1}(W_i,W_j)<\delta$ for all $i\neq j$. Similar notions have appeared as a way to improve Ramsey number bounds, e.g.\ for induced Ramsey numbers \cite{FS08} and for Ramsey numbers of bounded degree graphs \cite{CFS12}, but none of the previous notions have factored in edge-orderings.
  Note that a $(\alpha,\gamma,\delta,t)$-sparse edge-ordered graph is also $(\alpha',\gamma',\delta',t')$-sparse for $\alpha'\le \alpha$ and $\gamma'\le \gamma$ and $\delta'\ge\delta$, and $t'\le t$.
Lemma~\ref{lem:edge-1} shows that if $G_1$ has no copy of $H$, then it is $(n^{-2},\delta_1^n, \delta_1, 2)$-sparse for all sufficiently small $\delta_1$.\footnote{We actually show a stronger property, which we call $(n^{-2},\delta_1^n, \delta_1)$-sparse, and we use that stronger version for the multicolor case.}
The proof of Lemma~\ref{lem:edge-1} attempts to construct a copy of $H$ one vertex at a time and shows that the failure of this procedure implies the desired sparseness.

We then repeatedly apply the $(n^{-2},\delta_1^n,\delta_1, 2)$-sparseness of $G_1$ to show, in Lemma~\ref{lem:edge-3}, that, for all sufficiently small $\delta_2>0$ and $h=0,\dots,\ceil{\log n}$, we have $G_1$ is $(n^{-2^h+2}, \delta_2^{nh(1+o(1))}, \delta_2, 2^h)$-sparse.
The proof is by induction on $h$.
For the induction step, captured in Lemma~\ref{lem:edge-2}, we assume that the assertion is true for some $h$.
Given an interval $I$ and vertex subset $X$ we first apply the sparseness property guaranteed by Lemma~\ref{lem:edge-1} to obtain an interval $I'\subset I$ and two vertex subsets $W_Y$ and $W_Z$ such that the fraction pairs in $W_Y\times W_Z$ that form red $I'$-labeled edges is at most $\frac{\delta_2}{2^{h+2}}$, i.e.\ $d_{I'}\ind{G_1}(W_Y,W_Z)<\frac{\delta_2}{2^{h+2}}$.
Then we apply the induction hypothesis twice, first on $W_Y$, then on $W_Z$, carefully removing exceptional vertices before each application of the induction hypothesis, to obtain an interval $I''\subset I'$, a collection of $2^h$ pairwise disjoint subsets $W_1,\dots,W_{2^h}$ of $W_Y$, and a collection of $2^h$ pairwise disjoint subsets $W_{2^h+1},\dots,W_{2^{h+1}}$ of $W_Z$, such that $d_{I''}\ind{G_1}(W_i, W_j)<\delta_2$ for all $1\le i < j \le 2^{h+1}$.
This completes the induction.

Specializing the above to parameter $h=\ceil{\log n}$, vertex subset $X=[N]$, and interval $I=\binom{[N]}{2}$ gives that there is a large interval $I'\subset\binom{[N]}{2}$ and large pairwise disjoint vertex subsets $W_1,\dots,W_n\subset[N]$ such that, for all $1\le i < j \le n$, we have $d_{I'}\ind{G_1}(W_i,W_j) < \delta_2$.
If $\delta_2\ll |I'|/\binom{N}{2}$, then between any two $W_i$, almost all the $I'$-labeled edges are blue.
We show that, if $\delta_2$ is sufficiently small, then there is a blue copy of $H$ spanning $W_1,\dots,W_n$.

Take an equipartition of the interval $I'$ into $\binom{n}{2}$ consecutive intervals, and label the intervals by $J_{1,2},\dots, J_{n-1,n}$ according to the ordering of the edges of $H$.
Consider the unordered $n$-partite graph $F$ on vertex set $W_1\cup\cdots\cup W_n$ obtained by keeping between $W_i$ and $W_j$ exactly the edges in $G$ with label in $J_{i,j}$, and coloring the edges according to the edge-coloring of $G$.
By the choice of edges kept in $F$, an $n$-clique spanning $W_1,\dots,W_n$ forms a copy of $H$ in $G$.
Thus, to show $G$ has a monochromatic copy of $H$, it suffices to find a monochromatic blue $n$-clique in $F$ spanning $W_1,\dots,W_n$, which we do in Lemma~\ref{lem:edge-4} and outline below.

As $G$ is $\varepsilon$-regular, for $\alpha\defeq\alpha_{I'}/\binom{n}{2}$ and all $1\le i<j\le n$, the pair $(W_i, W_j)$ is $(\alpha,\varepsilon)$-regular in the graph $G_{J_{i,j}}$, and thus is $(\alpha,\varepsilon)$-regular in the graph $F$.\footnote{Since the intervals $J_{i,j}$ may differ in size by 1, we actually take $G$ to be $\frac{\varepsilon}{2}$-regular to guarantee the pair $(W_i,W_j)$ is $(\alpha,\varepsilon)$-regular in the graph $F$.}
The counting lemma (Lemma~\ref{lem:count}) can be used to approximate the total number of $n$-cliques in $F$, as well as the total number of cliques containing a red edge.
If parameters are chosen carefully, the former is larger than the latter, implying the existence of a clique with only blue edges.
In total, we use the following three applications of the counting lemma:
\begin{enumerate}
\item The number of $n$-cliques in $F$ is approximately
\begin{align}
  M\defeq \alpha^{\binom{n}{2}} \cdot \prod_{i=1}^{n} |W_i| .
\label{eq:outline-1}
\end{align}
\item Call an edge $w_iw_j$ \emph{normal} if, for all $k\neq i,j$, the codegree of $w_i$ and $w_j$ to $W_k$ is ``correct'', i.e.\ approximately $\alpha^2|W_k|$.
For a normal edge $w_iw_j$ and $k\neq i,j$, consider the common neighborhood $W_k'$ of $w_i$ and $w_j$ in $W_k$.
All pairs $(W_k', W_\ell')$ are $(\alpha,\varepsilon')$-regular for $\varepsilon'\approx \frac{\varepsilon}{\alpha^2}$.
Hence, we may apply the counting lemma again on the $n-2$ parts $(W_k')_{k\neq i,j}$ to show that the number of $n$-cliques extending each normal edge $w_iw_j$ is approximately $\frac{M}{\alpha|W_i||W_j|}$.
At most $\delta_2$ fraction of the pairs in $W_i\times  W_j$ form red edges for each $i$ and $j$, so the total number of $n$-cliques in $F$ containing a red normal edge is roughly at most
\begin{align}
  \sum_{1\le i<j\le n}^{} \left(\delta_2|W_i||W_j|\right)\cdot \frac{M}{\alpha|W_i| |W_j|} \ = \ \binom{n}{2}\delta_2\alpha^{-1}M. 
\label{}
\end{align}

\item Each non-normal edge $w_iw_j$ is in at most approximately
\begin{align}
  \alpha^{\binom{n-2}{2}} \prod_{i=1}^{n-2} |W_i| = \frac{M}{\alpha^{2n-2}|W_i| |W_j|}
\label{}
\end{align}
$n$-cliques by the counting lemma on the $n-2$ parts $(W_k)_{k\neq i,j}$. As any $W_i$ and $W_j$ are $(\alpha,\varepsilon)$-regular, one can show (Lemma~\ref{lem:reg-3}) that the number of non-normal edges between $|W_i|$ and $|W_j|$ is at most $4\varepsilon\alpha^{-1}n|W_i||W_j|$. 
Hence, the total number of $n$-cliques containing a non-normal edge, and in particular a red non-normal edge, is roughly at most
\begin{align}
  \sum_{1\le i<j\le n}^{} \left(4\varepsilon\alpha^{-1}n|W_i||W_j|\right)\cdot \frac{M}{\alpha^{2n-3}|W_i| |W_j|} 
  \ = \binom{n}{2} \cdot  4\varepsilon\alpha^{-(2n-2)}nM.
\label{}
\end{align}
\end{enumerate}
If $\delta_2$ and $\varepsilon$ are sufficiently small, we conclude that the graph $F$ has a clique with no red edges, so $G$ has a blue copy of $H$.

Now we show how to extend the above ideas to prove Theorem~\ref{thm:multi}, when there are $q > 2$ colors.
Suppose $G$ is an edge-ordered complete graph colored in $q$ colors and has no monochromatic copy of $H$ in any of the first $q-1$ colors.
For $1\le k \le q$, let $G_{k}$ denote the edge-ordered subgraph of $G$ consisting of the edges of color $k$, and let $G_{\le k}$ denote the edge-ordered subgraph consisting of the edges whose color is at most $k$.
We show, by induction on $k$, that for all $\delta_4$ sufficiently small, the edge-ordered graph $G_{\le k}$ is $(\alpha_{k},\gamma_{k},\delta_4, n)$-sparse for $\alpha_{k}=n^{-\Theta_k(n^k)}$ and $\gamma_k=\gamma_k(\delta_4)=\delta_4^{\Theta_k(n^k\log^kn)}$.
The base case $k=1$ was shown in the two-color case.
For the induction step, we show, in Lemma~\ref{lem:multi-1}, that if $G_{\le k}$ is $(\alpha_k, \gamma_k, \delta_4, n)$-sparse, and $G_{k+1}$ is $(n^{-2}, \delta_1^n,\delta_1)$-sparse for all sufficiently small $\delta_1$ (which we know is true from Lemma~\ref{lem:edge-1})\footnote{Here, we use the slightly stronger notion proved in Lemma~\ref{lem:edge-1}, which guarantees the two sets $W_1$ and $W_2$ to be subsets of a part of a given $n$-partition of the vertex subset $X$.}, then $G_{\le k+1}$ is $(\alpha_{k+1,1},\gamma_{k+1,1},\delta_4,2)$-sparse for some $\alpha_{k+1,1},\gamma_{k+1,1}=\gamma_{k+1,1}(\delta_4)$.
By the same induction on $h$ as in Lemma~\ref{lem:edge-3}, if $G_{\le k+1}$ is $(\alpha_{k+1,1},\gamma_{k+1,1},\delta_4,2)$-sparse, it is also $(\alpha_{k+1,h},\gamma_{k+1,h}, \delta_4, 2^h)$-sparse for $h=0,1,\dots,\ceil{\log n}$.
Setting $\alpha_{k+1}=\alpha_{k+1,\ceil{\log n}}$ and $\gamma_{k+1}=\gamma_{k+1,\ceil{\log n}}$ completes the induction on $k$.

Applying the above result to $k=q-1$ implies that $G_{\le q-1}$ is $(\alpha_{q-1},\gamma_{q-1}, \delta_4,n)$-sparse.
Choosing $\delta_4=n^{-\Theta_q(n^{q-1})}$ and applying this sparseness property to $X=[N]$ and $I=\binom{[N]}{2}$ gives that there is an interval $I'\subset\binom{[N]}{2}$ of size at least $\alpha_{q-1}\binom{N}{2}$ and vertex subsets $W_1,\dots,W_n$ of size at least $\gamma_{q-1}N$.
Our choice of $\delta_4$ guarantees that $\delta_4\ll \alpha_{q-1}$.
Hence, as in the two-color case, the counting argument in Lemma~\ref{lem:edge-4} guarantees a monochromatic copy of $H$ in the $q$th color.

\section{Two colors}
\label{sec:two}

\subsection{Sparseness in edge-ordered graphs}

In this subsection, we prove that, if an edge-ordered graph has no copy of a graph $H$, then it satisfies a certain sparseness property.
We use this to prove a $2^{O(n^3\log^2n)}$ bound on the edge-ordered Ramsey number (Proposition~\ref{prop:weak}).
In the following subsection, we improve the bound using a counting argument.
We start with a notion of sparseness in edge-ordered graphs.
\begin{definition}
  \label{def:sparse-1}
  For $\alpha,\gamma\in(0,1),\delta\in (0,\frac{1}{n})$, and a graph $H$ on vertices $1,\dots,n$, we say an edge-ordered graph $G$ is \emph{$(\alpha,\gamma,\delta)$-sparse with respect to $H$} if, for all intervals $I$ with $|I|\ge\alpha^{-1}$, all subsets of vertices $X$ with $|X|\ge \gamma^{-1}$, and all partitions $X_1\cup\cdots\cup X_n$ of $X$ with $|X_i| \ge \delta|X|$ for all $i=1,\dots,n$, the following is true.
  There exists an interval $I'\subset I$, an edge $(i,i')\in H$, and vertex subsets $W\subset X_i$ and $W'\subset X_{i'}$ such that 
  \begin{enumerate}
  \item[(i)] $|I'|\ge \alpha|I|$,
  \item[(ii)] $|W|$ and $|W'|$ are each at least $\gamma|X|$, and
  \item[(iii)] $d_{I'}(W, W') < \delta$.
  \end{enumerate}
  If the graph $H$ is omitted, we assume $H$ is a complete graph.
\end{definition}
The following lemma shows that an edge-ordered graph $G$ avoiding some edge-ordered complete graph must have the sparseness property above.
In Sections~\ref{sec:two} and \ref{sec:multi}, we apply Lemma~\ref{lem:edge-1} when $H$ is an edge-ordered complete graph, which is $(n-1)$-degenerate.
For such an $H$, Lemma~\ref{lem:edge-1} states that an edge-ordered graph $G$ with no copy of $H$ is $(n^{-2},\delta_1^n,\delta_1)$-sparse.
\begin{lemma}
  \label{lem:edge-1}
  If $n$ is a positive integer, $\delta_1$ is in the range $(0,\frac{1}{n})$, $H$ is a edge-ordered $d$-degenerate graph with vertex set $[n]$, and $G_1$ is an edge-ordered graph with no edge-ordered copy of $H$, then $G_1$ is $(n^{-2},\frac{1}{n-d}\delta_1^{d+1},\delta_1)$-sparse with respect to $H$.
\end{lemma}
\begin{proof}
  By the definition of degeneracy, the vertices of $H$ can be labeled $1,2,\dots,n$ such that, for every vertex $\ell$, the number of vertices $j<\ell$ adjacent to it is at most $d$.
  For $t$ and $i$ with $1\le t<i\le n$, let $D(t,i)$ denote the number of vertices $i'\le t$ such that vertex $i'$ is adjacent to vertex $i$ in $H$.

  Assume for contradiction that $G_1$ is not $(n^{-2},\frac{1}{n-d}\delta_1^{d+1},\delta_1)$-sparse.
  Then, there exists an interval $I$ with $|I|\ge n^2$, vertex subset $X$ with $|X|\ge (n-d)\delta_1^{-(d+1)}$, and a partition $X=X_1\cup\cdots\cup X_n$ with $|X_i|\ge \delta_1|X|$ for all $i$, such that the following holds: for every $i\neq i'$ and every $W\subset X_i$ and $W'\subset X_{i'}$ each of size at least $\frac{1}{n-d}\delta_1^{d+1}|X|$ and interval $I'\subset I$ of length at least $|I|/n^2$, we have $d_{I'}(W,W')\ge\delta_1$.
  Under these assumptions, we show there is a copy of $H$ in $G_1$, giving the desired contradiction.

  Create an equitable partition of $I$ by partitioning it into $|E(H)|$ consecutive intervals, and identify the intervals by $J_{i,i'}$ for $(i,i')\in E(H)$ according to the ordering of the edges of $H$.
  More precisely, if $I=[a+1,a']$ and the edge $(i,i')\in E(H)$ is the $j$th smallest edge in the ordering of the edges of $H$, then
  \begin{align}
    \label{eq:equitable}
    J_{i,i'}\defeq \left[a+1+\left\lfloor\frac{(a'-a)(j-1)}{|E(H)|}\right\rfloor, a+\left\lfloor\frac{(a'-a)j}{|E(H)|}\right\rfloor\right].
  \end{align}
  This ensures that, for all $(i,i')\in E(H)$, we have $|J_{i,i'}| \ge \floor{\frac{|I|}{|E(H)|}} > \frac{|I|}{n^2}$.

  \noindent
  \textbf{Claim.} For $t=0,1,\dots,n$, there exists vertices $v_1,\dots,v_t$ with $v_i \in X_i$ for $1 \leq i \leq t$ and sets $X_{t,t+1},X_{t,t+2},\dots,X_{t,n}$ with $X_{t,i} \subset X_i$ for $t<i \leq n$ such that the following three statements hold:
\begin{enumerate}
    \item[(i)] for all $(i,i')\in E(H)$ with $1\le i < i' \le t$, the pair $(v_i,v_{i'})$ is an edge of $G_1$ with label in $J_{i,i'}$,
    \item[(ii)] for all $(i,i')\in E(H)$ with $1\le i\le t < i'\le n$, every pair $(v_i,x)$ with $x\in X_{t,i'}$ is an edge of $G_1$ with label in $J_{i,i'}$, and
    \item[(iii)] $|X_{t,i}|\ge \delta_1^{D(t,i)+1}|X|$ for $i=t+1,\dots,n$.
\end{enumerate}

  To finish the proof of Lemma~\ref{lem:edge-1} from the claim, apply the claim for $t=n$ to obtain vertices $v_1,\dots,v_n$.
  Since the intervals $J_{i,i'}$ for $(i,i')\in E(H)$ are ordered according to the ordering of the edges of $H$, the vertices $v_1,\dots,v_n$ form an edge-ordered copy of $H$, as desired.

  We now prove the claim by induction on $t$.
  For $t=0$, the statement is true by taking $X_{0,i}=X_i$ for $i=1,\dots,n$.
  Properties (i) and (ii) are vacuous and property (iii) is satisfied as $|X_i|\ge \delta_1|X|=\delta_1^{D(0,i)+1}$ for $i=1,\dots,n$ by assumption.

  Now assume the claim is true for some $t-1$, where $t\le n$, so that there exist $v_1,\dots,v_{t-1}$ and sets $X_{t-1,t}, \dots, X_{t-1,n}$ for which properties (i), (ii), and (iii) hold.
  If $t=n$, we simply take $v_n$ to be an arbitrary element of $X_{n,n}$, which is nonempty by the induction hypothesis.
  Then property (i) is satisfied by the induction hypothesis, and properties (ii) and (iii) are vacuous.

  Now assume $t<n$.
  First, we show that, for all $i=t+1,\dots,n$ with $(t,i)\in E(H)$, less than $\frac{1}{n-t}|X_{t-1,t}|$ elements $v\in X_{t-1,t}$ have $\deg_{J_{t,i}}(v,X_{t-1,i}) < \delta_1|X_{t-1,i}|$.
  Suppose for contradiction this is false for some $i$.
  Let $W$ be the elements $v\in X_{t-1,t}$ such that $\deg_{J_{t,i}}(v,X_{t-1,i}) < \delta_1|X_{t-1,i}|$, so that $|W| \geq \frac{1}{n-t}|X_{t-1,t}|$.
  Let $W'=X_{t-1,i}$, so that
  \begin{align}
    \label{eq:edge-1-1}
    d_{J_{t,i}}(W,W') \ < \ \delta_1.
  \end{align}
  By the construction of $v_1,\dots,v_{t-1}, X_{t-1,t},\dots,X_{t-1,n}$ from the induction hypothesis, we have 
  \begin{align}
  |W|
  \ > \ \frac{1}{n-t}|X_{t-1,t}|
  \ \ge \ \frac{1}{n-t}\delta_1^{D(t-1,i)+1}|X|
  \ \ge \ \frac{1}{n-d}\delta_1^{d+1}|X|.
  \end{align}
  The last inequality uses that $D(t-1,i)\le d$ for all $i$.
  In the case $D(t-1,i) = d$, we know $t>d$ so $\frac{1}{n-t} > \frac{1}{n-d}$, and in the case $D(t-1,i)<d$, we use the bound $\frac{1}{n-t} \ge \delta_1\ge \frac{1}{n-d}\delta_1$.
  Additionally, the induction hypothesis gives
  \begin{align}
    |W'|=|X_{t-1,i}| > \delta_1^{d+1}|X| \ge \frac{1}{n-d}\delta_1^{d+1}|X|.
  \end{align}
  Also, $W\subset X_{t-1,t} \subset X_t$ and $W'\subset X_{t-1,i} \subset X_i$.
  By \eqref{eq:edge-1-1}, this choice of interval $I'=J_{t,i}$ and sets $W$ and $W'$ gives a contradiction of the assumption at the beginning of the lemma.

  By the preceding argument, for each $i=t+1,\dots,n$ with $(t,i)\in E(H)$, there are less than $\frac{1}{n-t}|X_{t-1,t}|$ vertices $v\in X_{t-1,t}$ such that $\deg_{J_{t,i}}(v,X_{t-1,i}) < \delta_1|X_{t-1,i}|$.
  Thus, there exists some $v_t\in X_{t-1,t}$ such that $\deg_{J_{t,i}}(v_t,X_{t-1,i}) \ge \delta_1|X_{t-1,i}|$ for all $i=t+1,\dots,n$ with $(t,i)\in E(H)$.
  Fix this choice $v_t$.
  For $i=t+1,\dots,n$ such that $(t,i)\in E(H)$, let $X_{t,i}$ be the vertices $x\in X_{t-1,i}$ such that edge $(v_t,x)$ has label in $J_{t,i}$.
  By the choice of $v_t$, we have $|X_{t,i}|\ge \delta_1|X_{t-1,i}|$ for such $i$.
  For $i=t+1,\dots,n$ such that $(t,i)\notin E(H)$, let $X_{t,i}=X_{t-1,i}$.

  We now show that $v_1,\dots,v_t, X_{t,t+1},\dots, X_{t,n}$ satisfy properties (i), (ii), and (iii) of the claim.
  For (i), when $(i,i')\in E(H)$ and $i,i'<t$, the edge between $v_i$ and $v_{i'}$ has label in $J_{i,i'}$ by the induction hypothesis.
  When $(i,i')\in E(H)$ and $i < t = i'$, the edge between $v_i$ and $v_t$ has label in $J_{i,t}$ because $v_t$ is in $X_{t-1,t}$ and all edges between $v_i$ and $X_{t-1,t}$ have label in $J_{i,t}$ by the induction hypothesis.
  For (ii), when $(i,i')\in E(H)$ and $i < t < i'$, all edges between $v_i$ and $X_{t-1,i'}$ have label in $J_{i,i'}$ by the induction hypothesis and $X_{t,i'}\subset X_{t-1,i'}$, so all edges between $v_i$ and $X_{t,i'}$ have label in $J_{i,i'}$.
  When $(i,i')\in E(H)$ and $i=t < i'$, all edges between $v_t$ and $X_{t,i'}$ have label in $J_{t,i'}$ by construction of $X_{t,i}$.
  For (iii), for all $i=t+1,\dots,n$ such that $(t,i)\in E(H)$, we have 
  \begin{align}
    |X_{t,i}|
    \ \ge \ \delta_1|X_{t-1,i}|
    \ \ge \ \delta_1^{1+D(t-1,i)+1}|X|
    \ = \ \delta_1^{D(t,i)+1}|X|,
  \end{align}
  where the second inequality is by the induction hypothesis.
  For all $i=t+1,\dots,n$ such that $(t,i)\notin E(H)$, we have
  \begin{align}
    |X_{t,i}|
    \ = \ |X_{t-1,i}|
    \ \ge \ \delta_1^{D(t-1,i)+1}|X|
    \ = \ \delta_1^{D(t,i)+1}|X|.
  \end{align}
  This completes the induction, proving the claim, and thus the lemma.
\end{proof}

Lemma~\ref{lem:edge-1} proves that an $H$-free edge-ordered graph $G$ is $(\alpha,\gamma,\delta)$-sparse for certain $(\alpha,\gamma,\delta)$.
Here, we introduce another notion of sparseness, $(\alpha,\gamma,\delta,t)$-sparse.
In this definition, $(\alpha,\gamma,\delta,2)$-sparseness is implied by $(\alpha,\gamma,\delta)$-sparseness, but is only slightly weaker.
For our two-color result, Theorem~\ref{thm:edge-1}, we only need that $H$-freeness implies $(\alpha,\gamma,\delta,2)$-sparseness.
However the extra strength of $(\alpha,\gamma,\delta)$-sparseness is needed for the multicolor generalization, Theorem~\ref{thm:multi}.
\begin{definition}
  \label{def:sparse-2}
  For $\alpha,\gamma,\delta\in(0,1)$ and positive integer $t$, we say an edge-ordered graph $G$ is \emph{$(\alpha,\gamma,\delta,t)$-sparse} if, for all intervals $I$ with $|I|\ge\alpha^{-1}$ and all subsets of vertices $X$ with $|X|\ge \gamma^{-1}$, there exists an interval $I'\subset I$ and pairwise disjoint subsets of vertices $W_1,\dots,W_t$ such that 
  \begin{enumerate}
  \item[(i)] $|I'|\ge \alpha|I|$,
  \item[(ii)] $|W_i|\ge \gamma|X|$ for $i=1,\dots,t$, and
  \item[(iii)] $d_I(W_i,W_j) < \delta$ for $1\le i < j\le t$.
  \end{enumerate}
\end{definition}
The following facts immediately follow from Definitions~\ref{def:sparse-1} and \ref{def:sparse-2}.
\begin{fact}
  \label{fact:sparse-1}
  Let $\alpha,\gamma,\delta>0$.
  An $(\alpha,\gamma,\delta)$-sparse edge-ordered graph is also $(\alpha,\gamma,\delta,2)$-sparse.
\end{fact}
\begin{fact}
  \label{fact:sparse-2}
  Let $\alpha,\alpha',\gamma,\gamma',\delta,\delta'>0$ be such that $\alpha\ge \alpha'$ and $\gamma\ge \gamma'$ and $\delta\le\delta'$, and let $t$ and $t'$ be integers with $t\ge t'\ge 1$.
  An $(\alpha,\gamma,\delta)$-sparse edge-ordered graph is also $(\alpha',\gamma',\delta')$-sparse, and an $(\alpha,\gamma,\delta,t)$-sparse edge-ordered graph is also $(\alpha',\gamma',\delta',t')$-sparse.
\end{fact}

We use the following simple proposition twice in the proof of the next lemma.
\begin{proposition}
  \label{prop:edge-1}
  Let $c\ge 1$ and $\delta\in(0,1)$.
  In a graph $F$, if $X$ and $Y$ are disjoint sets of vertices such that $d\ind{F}(X,Y)\le\delta$ then the number of vertices $x$ in $X$ such that $\deg\ind{F}(x, Y)\ge c\delta|Y|$ is at most $|X|/c$.
\end{proposition}
\begin{proof}
  Call a vertex $x\in X$ \emph{bad} if $\deg\ind{F}(x,Y)\ge c\delta|Y|$.
  The number of edges between $X$ and $Y$ is at most $\delta|X||Y|$, but it is also at least $c\delta|Y|$ times the number of bad vertices.
  Hence, the number of bad vertices is at most $\frac{\delta|X||Y|}{c\delta|Y|} = |X|/c$.
\end{proof}

The following lemma gives us sparseness properties with larger values of $t$.
\begin{lemma}
\label{lem:edge-2}
Let $\alpha,\alpha',\gamma,\gamma',\delta>0$, let $t$ be a positive integer, and let $\delta'=\frac{\delta}{4t}$.
Suppose that $G_1$ is an edge-ordered graph that is both $(\alpha,\gamma,\delta,t)$-sparse and $(\alpha',\gamma',\delta',2)$-sparse. 
Then $G_1$ is $(\alpha^2\alpha',\half\gamma\gamma',\delta,2t)$-sparse.
\end{lemma}
\begin{proof}
  Let $I$ be an interval with at least $(\alpha^2\alpha')^{-1}$ integers.
  Let $X$ be a subset of the vertices of $G_1$ of size at least $(\half \gamma\gamma')^{-1}$.
  Applying the $(\alpha',\gamma',\delta',2)$-sparse property of $G_1$ to the interval $I$ and set $X$ gives an interval $I_1\subset I$ and subsets $W_Y$ and $W_Z$ of $X$ with the following properties: 
  \begin{enumerate}
  \item[(i)] $|I_1|\ge \alpha'|I|$,
  \item[(ii)] $W_Y$ and $W_Z$ are of size at least $\gamma'|X|$, and 
  \item[(iii)] $d_{I_1}(W_Y,W_Z)<\delta'$.
  \end{enumerate}
  By Proposition~\ref{prop:edge-1} on the parameters $c=2$ and $\delta'$ and the graph $G_{I_1}$ with subsets $W_Y$ and $W_Z$, there exists a subset $W_Y'\subset W_Y$ of size at least $|W_Y| - \frac{1}{2}|W_Y| = \frac{1}{2}|W_Y|$ such that, for all $y\in W_Y'$, we have
  \begin{align}
    \deg_{I_1}(y, W_Z) < 2\delta'|W_Z|.
  \label{eq:edge-11}
  \end{align}
  Apply the $(\alpha,\gamma,\delta,t)$-sparse property of $G_1$ to the interval $I_1$ and the set of vertices $W_Y'$ to obtain an interval $I_2\subset I_1$ and disjoint subsets $Y_1,\dots,Y_{t}$ of $W_Y'$ with the following properties: 
  \begin{enumerate}
  \item[(i)] $|I_2|\ge \alpha|I_1|$, 
  \item[(ii)] $Y_1,\dots,Y_{t}$ each have size at least $\gamma|W_Y'|$, and 
  \item[(iii)] for all $i$ and $j$ with $1\le i < j\le t$, we have
  \begin{align}
    d_{I_2}(Y_i, Y_j) < \delta.
  \label{eq:edge-13}
  \end{align}
  \end{enumerate}
  For all $i=1,\dots,t$, we have $Y_i\subset W_Y'$.
  By \eqref{eq:edge-11} and the fact that $I_2\subset I_1$, we obtain $d_{I_2}(Y_i,W_Z) < 2\delta'$.
  For any $i=1,\dots,t$, by Proposition~\ref{prop:edge-1} on the parameters $c=2t$ and $2\delta'$ and the graph $G_{I_2}$ with sets $W_Z$ and $Y_i$, at most a $\frac{1}{2t}$ fraction of the vertices $z\in W_Z$ satisfy $\deg_{I_2}(z, Y_i) \ge 4t\delta'|Y_i| = \delta|Y_i|$.
  As there are at most $t$ choices of $i$, there exists a subset $W_Z'\subset W_Z$ of size at least $|W_Z|-t\cdot\frac{1}{2t}|W_Z|=\frac{1}{2}|W_Z|$ such that, for all $z\in W_Z'$ and all $i=1,\dots,t$, we have
  \begin{align}
    \deg_{I_2}(z,Y_i) < \delta|Y_i|.
  \label{eq:edge-15}
  \end{align}

  Apply the $(\alpha,\gamma,\delta,t)$-sparse property of $G_1$ to the interval $I_2$, and the set of vertices $W_Z'$ to obtain an interval $I_3\subset I_2$ and disjoint subsets $Z_1,\dots,Z_{t}$ of $W_Z'$ with the following properties: 
  \begin{enumerate}
  \item[(i)] $|I_3|\ge \alpha|I_2|$, 
  \item[(ii)] $Z_1,\dots,Z_{t}$ each have size at least $\gamma|W_Z'|$, and 
  \item[(iii)] for all $i$ and $j$ with $1\le i < j\le t$, we have
  \begin{align}
    d_{I_3}(Z_i, Z_j) < \delta.
  \label{eq:edge-16}
  \end{align}
  \end{enumerate}
  For all $i$ and $j$ between 1 and $t$, we have $Z_j\subset W_Z'$, and combining with \eqref{eq:edge-15} and $I_3\subset I_2$ gives
  \begin{align}
    d_{I_3}(Z_j, Y_i) < \delta.
  \label{eq:edge-17}
  \end{align}
  Observe that
  \begin{align}
    |I_3|
    \ &\ge \ \alpha|I_2| 
    \ \ge \ \alpha^2|I_1| 
    \ \ge \ \alpha^2\alpha'|I|.
  \label{eq:edge-18-a}
  \end{align}
  Additionally, for all $i=1,\dots,t$,
  \begin{align}
    \label{eq:edge-18-b}
    |Y_i|
    \ &\ge \ \gamma |W_Y'| 
    \ \ge \  \half \gamma |W_Y|  
    \ \ge \  \half \gamma\gamma' |X|,\\
    \label{eq:edge-18-c}
    |Z_i|
    \ &\ge \ \gamma |W_Z'| 
    \ \ge \  \half \gamma |W_Z|  
    \ \ge \  \half \gamma\gamma' |X|.
  \end{align}
  Set $I'=I_3$, and for $i=1,\dots,t$, set $W_i = Y_i$ and $W_{t+i}=Z_i$.
  By \eqref{eq:edge-13}, \eqref{eq:edge-16}, \eqref{eq:edge-17}, \eqref{eq:edge-18-a}, \eqref{eq:edge-18-b}, and \eqref{eq:edge-18-c}, interval $I'$ and sets $W_1,\dots,W_{2t}$ satisfy (i), (ii), (iii) in Definition~\ref{def:sparse-2} for $(\alpha^2\alpha', \half\gamma\gamma',\delta,2t)$-sparseness.
  This holds for any $X$ and $I$, so $G_1$ is $(\alpha^2\alpha',\half\gamma\gamma',\delta,2t)$-sparse, as desired.
\end{proof}

Lemma~\ref{lem:edge-2}, iterated $\ceil{\log n}$ times, together with Lemma~\ref{lem:edge-1}, implies the following corollary.
\begin{lemma}
  \label{lem:edge-3}
  Let $\delta_2\in (0,\frac{1}{n})$ and $H$ be an edge-ordered $d$-degenerate graph.
  Let $G_1$ be an edge-ordered graph with no edge-ordered copy of $H$.
  For all integers $h\ge 0$, we have that $G_1$ is $(\alpha_h, \gamma_h, \delta_2, 2^h)$-sparse, where $\alpha_h\defeq n^{-(2^{h+1}-2)}$ and $\gamma_h\defeq(\frac{\delta_2}{2^{h+2}(n-d)})^{h(d+1)}$.
\end{lemma}
\begin{proof}
  The base case $h=0$ states that $G_1$ is $(1,1,\delta_2,1)$-sparse, which is true by taking $I'=I$ and $X_1=X$ for any $I$ and $X$: properties (i) and (ii) in Definition~\ref{def:sparse-2} are satisfied and (iii) is vacuous.

  For the induction step, suppose $G_1$ is $(\alpha_h, \gamma_h, \delta_2, 2^h)$-sparse for some integer $h$.
  Set $\delta_1=\frac{\delta_2}{2^{h+2}}$.
  As $\delta_1 < \frac{1}{n}$, we may apply Lemma~\ref{lem:edge-1} to obtain that $G_1$ is $(n^{-2},\frac{1}{n-d}\delta_1^{d+1},\delta_1)$-sparse.
  Applying Lemma~\ref{lem:edge-2} for $\alpha=\alpha_h, \gamma=\gamma_h, \delta=\delta_2, t=2^h, \alpha'=n^{-2}, \gamma'=\frac{1}{n-d}\delta_1^{d+1}$, and $\delta'=\frac{\delta_2}{4t}=\delta_1$, we have that $G_1$ is $(\alpha_h^2n^{-2},\frac{1}{2(n-d)} \delta_1^{d+1}\gamma_h,\delta_2,2^{h+1})$-sparse.
  As 
  \begin{align}
    \alpha_h^2n^{-2} \ &= \   n^{-(2^{h+2}-2)} = \alpha_{h+1}, \nonumber\\
    \frac{1}{2(n-d)}\delta_1^{d+1}\gamma_h 
    \ &= \   \frac{\delta_2^{d+1}}{2^{(h+2)(d+1)}\cdot 2(n-d)}\left(\frac{\delta_2}{2^{h+2}(n-d)}\right)^{h(d+1)}  \nonumber\\
    \ &< \ \left(\frac{\delta_2}{2^{h+3}(n-d)}\right)^{(h+1)(d+1)} 
    \ = \ \gamma_{h+1},
  \end{align}
  we have, by Fact~\ref{fact:sparse-2}, $G_1$ is $(\alpha_{h+1}, \gamma_{h+1}, \delta_2, 2^{h+1})$-sparse.
  This completes the induction, proving the corollary.
\end{proof}
All graphs on $n$ vertices are $(n-1)$-degenerate. 
Setting $h=\ceil{\log t}$ and $d=n-1$ and using Fact~\ref{fact:sparse-2} and the bound $2^{\ceil{\log t}}< 2t$ gives the following corollary of Lemma~\ref{lem:edge-3}.
\begin{corollary}
  \label{cor:edge-3}
  Let $\delta_2\in (0,\frac{1}{n})$, $t\ge 1$, and $H$ be an edge-ordered complete graph.
  Let $G_1$ be an edge-ordered graph with no copy of $H$.
  Then $G_1$ is $(n^{-4t+2}, (\frac{\delta_2}{8t})^{n\ceil{\log t}}, \delta_2, t)$-sparse.
\end{corollary}

In the proof of Theorem~\ref{thm:edge-1}, we assume a two-colored edge-ordered $\varepsilon$-regular graph $G$ has no red copy of $H$. 
Corollary~\ref{cor:edge-3} gives a sparseness property on the red subgraph $G_1$ of $G$.
This sparseness property yields an interval $I'$ and $n$ sets $W_1,\dots,W_n$ such that, between any two $W_i$, the fraction of pairs that are red $I'$-labeled edges is small.
We now show that there is a blue copy of $H$ spanning $W_1,\dots,W_n$.

One way to show this is to apply Lemma~\ref{lem:edge-1} again.
This method gives an upper bound on the edge-ordered Ramsey number of the form $2^{O(n^3\log^2n)}$.
We first show this method, and then present an alternative approach that improves the bound, but is a bit longer and uses additional ideas.
\begin{proposition}
  \label{prop:weak}
  For any edge-ordered graph $H$ on $n$ vertices, we have $r_{\edge}(H)\le 2^{180n^3\log^2n}$.
\end{proposition}
\begin{proof}
If $n\le 3$, the edge-ordered Ramsey number is simply the usual Ramsey number, so assume $n\ge 4$.
Let $N=2^{180n^3\log^2n}$, $\varepsilon=2\cdot 2^{-30n^3\log^2n}$, $\delta_1=2^{-12n^2\log n}$, and
 $\delta_2=2^{-5n\log n}$.
Let $G$ be an $\varepsilon$-regular edge-ordered complete graph on $N$ vertices, which exists by Lemma~\ref{lem:edge-0}.
Assume for contradiction that $G$ is colored such that the red subgraph $G_1$ and the blue subgraph $G_2$ each have no monochromatic copy of $H$.
By Corollary~\ref{cor:edge-3}, there exist an interval $I$ of size at least $n^{-4n+2}\binom{N}{2}$ and subsets $W_1,\dots,W_n$ of size at least $(\frac{\delta_1}{8n})^{n\ceil{\log n}}N > (\frac{\delta_1}{8n})^{2n\log n}N > 2^{-25n^3\log^2 n}N$ such that, for any $1\le i<j\le n$, we have
\begin{align}
  \label{eq:weak-0}
  d_I\ind{G_1}(W_i,W_j) < \delta_1.
\end{align}
It follows that there are equal-sized pairwise disjoint sets $X_1,\dots,X_n$ of size $2^{-25n^3\log^2 n}N$ such that
\begin{align}
  \label{eq:weak-1}
  d_I\ind{G_1}(X_i,X_j) < n^2\delta_1.
\end{align}
Indeed, the above holds by taking $X_i\subset W_i$ to be a uniformly random subset of $W_i$ of the appropriate size, and using $\E[d_I\ind{G_1}(X_i,X_j)] = d_I\ind{G_1}(W_i,W_j)$, Markov's inequality, and a union bound.
Lemma~\ref{lem:edge-1} implies that $G_2$ is $(n^{-2}, \delta_2^n,\delta_2)$-sparse.
Applying this sparseness to the set $X=X_1\cup\cdots\cup X_n$ and the interval $I$ gives an interval $I'\subset I$ and subsets $W_1'$ and $W_2'$ such that $W_1'\subset X_i$ and $W_2'\subset X_j$ for some $i\neq j$, and 
\begin{itemize}
\item [(i)] $|I'|\ge n^{-2}|I|\ge n^{-4n}\binom{N}{2}$,
\item [(ii)] $|W_1'|\ge \delta_2^n|X|$ and $|W_2'|\ge \delta_2^n|X|$, and
\item [(iii)] $d_{I'}\ind{G_2}(W_1',W_2') <  \delta_2$.
\end{itemize}
By \eqref{eq:weak-1} and (ii), we have
\begin{align}
  d_{I'}\ind{G_1}(W_1',W_2')
  \ \le \ d_I\ind{G_1}(W_1',W_2')
  \ \le \ \delta_2^{-2n}d_I\ind{G_1}(X_i,X_j)
  \ < \ n^2\delta_2^{-2n}\delta_1.
\label{}
\end{align}
This means the total density of $I'$-labeled edges between $W_1'$ and $W_2'$ satisfies
\begin{align}
  d_{I'}\ind{G}(W_1',W_2')
  \ = \ d_{I'}\ind{G_1}(W_1',W_2') + d_{I'}\ind{G_2}(W_1',W_2')
  \ < \ \delta_2^{-2n}\delta_1 + \delta_2
  \ < \ 2\cdot 2^{-5n\log n}.
  \label{eq:edge-3-3}
\end{align}
Since $|W_1'|\ge \delta_2^n|X|\ge \delta_2^n\cdot2^{-25n^3\log^2n} > \varepsilon N$ and similarly, $|W_2'|> \varepsilon N$, we have, by $\varepsilon$-regularity of $G$, that 
\begin{align}
  d_{I'}\ind{G}(W_1',W_2')
  \ > \ \frac{|I'|}{\binom{N}{2}} - \varepsilon
  \ > \ \frac{1}{2} n^{-4n}.
\end{align}
This contradicts \eqref{eq:edge-3-3} and completes the proof.
\end{proof}

\subsection{Applying the counting lemma}

We now give an alternative approach to finding a blue copy of $H$ that improves on the bound in Proposition~\ref{prop:weak}, which ultimately gives Theorem~\ref{thm:edge-1}.
Instead of iterating the sparseness property, we instead use a counting approach.
The following lemma is the key step.
\begin{lemma}
\label{lem:edge-4}
  Let $n\ge 4$ and  $\alpha,\delta_3,\varepsilon>0$ satisfy $\alpha<\frac{1}{n}$, $\delta_3<\frac{\alpha}{8n^2}$, and $\varepsilon<\frac{1}{8}\alpha^{2n+1}$.
  Let $F$ be an $n$-partite graph with parts $W_1,\dots,W_n$.
  Suppose that, for all $i \not = j$, the pair $(W_i,W_j)$ is $(\alpha,\varepsilon)$-regular.
  If some of the edges of $F$ are colored red such that, for any distinct $i$ and $j$, at most a $\delta_3$ fraction of the pairs between $W_i$ and $W_j$ are red edges.
  Then $F$ contains an $n$-clique with no red edges.
\end{lemma}
\begin{proof}
  Let
  \begin{align}
    M\defeq \alpha^{\binom{n}{2}} \cdot \prod_{i=1}^{n} |W_i| .
  \label{}
  \end{align}
  From the assumptions, $(\alpha/2)^n > \alpha^{2n} > \varepsilon$.
  Hence, we can apply Lemma~\ref{lem:count} on the subsets $W_1,\dots,W_n$ and the parameters $p_{i,j}'=\alpha$ for all $1\le i < j\le n$ and $\varepsilon'=\varepsilon$.
  This gives that the number of $n$-cliques in $F$ is at least
  \begin{align}
    \label{eq:edge-25}
    \left( 1 - \frac{4\varepsilon n}{\alpha^n} \right)\cdot \alpha^{\binom{n}{2}}\cdot \prod_{i=1}^{n}  |W_i| 
    \ = \ \left( 1- \frac{4\varepsilon n}{\alpha^n} \right)\cdot  M.
  \end{align}
  For $i \not = j$, call an edge $w_iw_j$ with $w_i\in W_i$ and $w_j\in W_j$ \emph{normal} if, for every $k \in [n] \setminus \{i,j\}$, we have
  \begin{align}
    \deg (w_i,w_j, W_k)\in (\alpha^2\pm 2\varepsilon)|W_k|.
  \label{}
  \end{align}
  As the pairs $(W_i,W_k)$ and $(W_j,W_k)$ are $(\alpha,\varepsilon)$-regular, we may apply Lemma~\ref{lem:reg-3} on the sets $X=W_i$, $Y=W_j$ and $Z=W_k$ for $k\neq i,j$ to obtain that the number of non-normal edges between $W_i$ and $W_j$ is at most $(n-2)\cdot 4\varepsilon \alpha^{-1}|W_i| |W_j| < 4\varepsilon\alpha^{-1}n|W_i||W_j|$. 

  We now show that $F$ has, relative to the number $M$, few $n$-cliques containing a red edge.
  We do so by caseworking on whether a given red edge is normal.
  Suppose $w_{n-1}w_{n}$ is a red edge with $w_{n-1}\in W_{n-1}$ and $w_n\in W_n$.

  \textbf{Case 1:} $w_{n-1}w_{n}$ is red and normal.
  For $i=1,\dots,n-2$, let $W_i'$ be the common neighbors of $w_{n-1}$ and $w_n$ in $W_i$. 
  As edge $w_{n-1}w_n$ is normal, we have $|W_i'|\in (\alpha^2\pm 2\varepsilon)|W_i|$ for $i=1,\dots,n-2$.
  Because the pair $(W_i,W_j)$ is $(\alpha,\varepsilon)$-regular for all $1\le i<j\le n-2$, the pair $(W_i',W_j')$ is $(\alpha, \frac{\varepsilon}{\alpha^2-2\varepsilon})$-regular by Lemma~\ref{lem:reg-2}.
  The tuples $(w_1,\dots,w_{n-2})$ with $w_i\in W_i$ for $i=1,\dots,n-2$ such that $(w_1,\dots,w_n)$ forms an $n$-clique are exactly those tuples such that $w_i\in W_i'$ for $i=1,\dots,n-2$ and such that $(w_1,\dots,w_{n-2})$ forms an $(n-2)$-clique. 
  As $(\alpha/2)^{n-2} > \frac{\varepsilon}{\alpha^2-2\varepsilon}$, we can apply Lemma~\ref{lem:count} on the sets $W_1',\dots,W_{n-2}'$ with the parameters $p_{i,j}=\alpha$ for all $1\le i<j\le n-2$ and $\varepsilon' = \frac{\varepsilon}{\alpha^2-2\varepsilon}$.
  Hence, the number of such $(n-2)$-cliques is at most 
  \begin{align}
    \left( 1 + \frac{4\cdot\frac{\varepsilon}{\alpha^2-2\varepsilon}\cdot n}{\alpha^{n-2}} \right)\cdot \alpha^{\binom{n-2}{2}}\cdot \prod_{i=1}^{n-2}  |W_i'| 
    \ &< \  2\cdot \alpha^{\binom{n-2}{2}}\cdot \prod_{i=1}^{n-2}  (\alpha^2+2\varepsilon)|W_i|  \nonumber\\
    \ &< \ 4\cdot \alpha^{\binom{n}{2}-1}\cdot \prod_{i=1}^{n-1} |W_i|
    \ = \  \frac{4M}{\alpha|W_{n-1}||W_n|}.
  \label{}
  \end{align}
  In the first inequality, we used that $\frac{4\varepsilon n}{\alpha^{n-2}(\alpha^2-2\varepsilon)} < \frac{8\varepsilon n}{\alpha^n} < 8\alpha^{n+1}n < 1$ and $|W_i'|\le (\alpha^2+2\varepsilon)|W_i|$. 
  In the second inequality, we used $(\alpha^2+2\varepsilon)^{n-2} \le \alpha^{2(n-2)}e^{2\varepsilon(n-2)/\alpha^2} < 2\alpha^{2(n-2)}$.
  By construction of $W_1,\dots,W_n$, the total number of red edges between $W_{n-1}$ and $W_n$ is at most $\delta_3|W_{n-1}| |W_n|$, so the total number of $n$-cliques containing a red normal edge between $W_{n-1}$ and $W_n$ is at most $\delta_3 \cdot 4\alpha^{-1}M < \frac{M}{2n^2}$.

  \textbf{Case 2:} $w_{n-1}w_n$ is red and non-normal.
  The number of $n$-cliques containing $w_{n-1}$ and $w_n$ is bounded by the number of $(n-2)$-cliques spanning $W_1,\dots,W_{n-2}$.
  As $(\alpha/2)^{n-2} > \varepsilon$, we can apply Lemma~\ref{lem:count} on the sets $W_1,\dots,W_{n-2}$ with the parameters $p_{i,j}'=\alpha$ for all $1\le i < j \le n-2$ and $\varepsilon'=\varepsilon$.
  Hence, the number of $n$-cliques $(w_1,\dots,w_n)$ extending $w_{n-1}$ and $w_n$ is at most 
  \begin{align}
    \left( 1 + \frac{4\varepsilon n}{\alpha^n} \right)\cdot \alpha^{\binom{n-2}{2}}\cdot \prod_{i=1}^{n-2} |W_i|
    \ < \ 2\cdot \frac{M}{\alpha^{2n-3}|W_{n-1}||W_n|}.
  \end{align}
  As argued above, the number of non-normal edges in $W_{n-1}\times W_n$ is at most $4\varepsilon\alpha^{-1}n|W_{n-1}||W_n|$, so the total number of $n$-cliques with a non-normal edge between $W_{n-1}$ and $W_n$ is at most $4\varepsilon\alpha^{-1}n\cdot 2\alpha^{-(2n-3)}M< n\alpha^3M<\frac{M}{n^2}$. 

  In total, the number of $n$-cliques in $F$ containing a red edge between $W_{n-1}$ and $W_n$ is at most $(\frac{1}{2n^2}+\frac{1}{n^2})\cdot M=\frac{3M}{2n^2}$.
  By a symmetric argument, this number also bounds the number of $n$-cliques containing a red edge between $W_i$ and $W_j$ for any $i$ and $j$ with $1\le i<j\le n$.
  Thus, the total number of $n$-cliques of $F$ containing any red edge is at most $\binom{n}{2}\cdot \frac{3M}{2n^2} < \frac{3}{4}M$, which is less than the quantity in \eqref{eq:edge-25}, as $1-\frac{4\varepsilon n}{\alpha^n} > 1 - 4\alpha^n n  > 1-\frac{1}{4} = \frac{3}{4}$.
  Hence there exists an $n$-clique of $F$ containing no red edge.
\end{proof}

We now can prove Theorem~\ref{thm:edge-1}.
\begin{proof}[Proof of Theorem~\ref{thm:edge-1}]
  If $n\le 3$, the edge-ordered Ramsey number is simply the usual Ramsey number, so assume $n\ge 4$.
  Let $N=2^{100n^2\log^2n}$, $\varepsilon=2^{-16n^2\log^2n}$, and $\delta_2 = 2^{-6n\log n}$, so that $N\ge (\varepsilon/4)^6$.
  Let $G$ be an $\frac{\varepsilon}{2}$-regular edge-ordered complete graph on $N$ vertices, which exists by Lemma~\ref{lem:edge-0}.
  Consider a red-blue edge-coloring of $G$.
  Suppose there exists some edge-ordered graph $H'$ such that $G$ has no red copy of $H'$.
  We show $G$ has a blue copy of every edge-ordered complete graph $H$ on $n$ vertices.
  Fix such an $H$.

  Let $G_1$ be the edge-ordered graph obtained from $G$ by keeping the red edges of $G$ and their corresponding edge labels, 
  By Corollary~\ref{cor:edge-3} with parameter $\delta_2$, because the edge-ordered graph $G_1$ contains no copy of $H'$, the edge-ordered graph $G_1$ is $(n^{-4n+2},(\frac{\delta_2}{8n})^{n\ceil{\log n}}, \delta_2, n)$-sparse.
  Hence, applying the sparsity property on interval $I=[1,\binom{N}{2}]$ and vertex subset $X=[N]$ gives an interval $I'\subset I$ and sets $W_1,\dots,W_n$ such that
  \begin{itemize}
  \item [(i)] $|I'|\ge n^{-4n+2}\binom{N}{2}$, 
  \item [(ii)] $|W_i|\ge (\frac{\delta_2}{8n})^{n\ceil{\log n}} > (2^{-6n\log n - \log n - 3})^{2n\log n} > \varepsilon N$, and 
  \item [(iii)] $d_{I'}\ind{G_1}(W_i, W_j) < \delta_2$ for all $i$ and $j$.
  \end{itemize}
  Let $J_{1,2},\dots,J_{n-1,n}$ be $\binom{n}{2}$ intervals forming an equitable partition of $I'$, ordered according to the order of the edges of $H$, as in \eqref{eq:equitable}.
  Let $\alpha=\frac{|J_{1,2}|}{\binom{N}{2}}$.
  In this way, $\alpha > \frac{|I'|}{n^2\binom{N}{2}} \ge n^{-4n}$ and $|\frac{|J_{i,j}|}{\binom{N}{2}} - \alpha| \le \frac{1}{\binom{N}{2}} < \frac{\varepsilon}{2}$ for all $1\le i < j\le n$.
  Let $F$ be an unordered $n$-partite graph with parts $W_1,\dots,W_n$, such that the edges between $W_i$ and $W_j$ are precisely the edges of $G$ between $W_i$ and $W_j$ whose label is in $J_{i,j}$.
  In this way, an $n$-clique in $F$ forms a copy of $H$ in $G$.
  Color the edges of $F$ according to the coloring of $G$.
  For all $i$ and $j$ with $1\le i < j\le n$, we have $J_{i,j}\subset I'$, so, in $F$, between any distinct $W_i$ and $W_j$, the fraction of pairs that are red edges is at most $\delta_2$.
  Recall that $|W_i|> \varepsilon N$ for all $i$, so, as $G$ is $\frac{\varepsilon}{2}$-regular, for all $1\le i < j\le n$, the pair $(W_i,W_j)$ is $(\frac{|J_{i,j}|}{\binom{N}{2}},\frac{\varepsilon}{2})$-regular in the graph $G_{J_{i,j}}$, and hence in the graph $F$.
  By Fact~\ref{fact:reg-1}, every pair $(W_i,W_j)$ is $(\alpha,\varepsilon)$-regular in the graph $F$.
  As $\alpha \le \frac{1}{\binom{n}{2}} < \frac{1}{n}$, $\delta_2<\frac{\alpha}{8n^2}$, and $\varepsilon < \frac{1}{8}\alpha^{2n+1}$, we may apply Lemma~\ref{lem:edge-4} with the parameters $\alpha,\delta_3=\delta_2$, and $\varepsilon$, the graph $F$ with parts $W_1,\dots,W_n$, and the coloring inherited from $G$.
  This tells us that $F$ contains an $n$-clique with no red edges.
  Hence $F$ has a blue $n$-clique, so $G$ has a blue copy of $H$.
\end{proof}


\section{More than two colors}
\label{sec:multi}

In this section, we prove Theorem~\ref{thm:multi}.
Throughout the section, when $G$ is a $q$-colored edge-ordered complete graph, we identify the colors as $1,\dots,q$.
Further, for $k=1,\dots,q$, let $G_k$ denote the edge-ordered subgraph of $G$ where we keep only the edges of color $k$, and we let $G_{\le k}$ denote the edge-ordered graph where we keep only the edges whose color is at most $k$.
\begin{lemma}
  \label{lem:multi-1}
  Let $q,k$ be positive integers with $q>k$ and $\alpha,\alpha',\gamma,\gamma',\delta,\delta'\in(0,1)$.
  Let $G$ be a $q$-colored edge ordered complete graph.
  If $G_{\le k}$ is $(\alpha,\gamma,\delta,n)$-sparse and $G_{k+1}$ is $(\alpha',\gamma',\delta')$-sparse, then $G_{\le k+1}$ is $(\alpha\alpha',\gamma\gamma',n^2(\gamma')^{-2}\delta + \delta', 2)$-sparse.
\end{lemma}
\begin{proof}
  Fix an interval $I$ of length at least $(\alpha\alpha')^{-1}$ and a vertex subset $X$ of size at least $(\gamma\gamma')^{-1}$.
  Applying the $(\alpha,\gamma,\delta,n)$-sparse property of $G_{\le k}$ to the interval $I$ and set $X$ to obtain an interval $I_1$ and sets $X_1,\dots,X_n$ such that
  \begin{enumerate}
  \item[(i)] $|I_1|\ge \alpha|I| \ge (\alpha')^{-1}$,
  \item[(ii)] $|X_i|\ge \gamma|X| \ge (\gamma')^{-1}$ for all $i=1,\dots,n$, and
  \item[(iii)] $d\ind{G_{\le k}}_{I_1}(X_i,X_j) < \delta$ for all $i\neq j$.
  \end{enumerate}
  By considering random subsets of $X_1,\dots,X_n$ of size equal to $\gamma|X|$, we have that there exist sets $X_1',\dots,X_n'$ each of the same size $\gamma|X|$ such that, for all $i$ and $j$ with $1\le i < j\le n$, we have 
  \begin{align}
    \label{eq:multi-0}
    d\ind{G_{\le k}}_{I_1}(X_i',X_j')<n^2\delta.
  \end{align}
  Indeed, over the randomness of the $X_i'$s, any density bound \eqref{eq:multi-0} is violated with probability at most $\frac{1}{n^2}$ by Markov's inequality on the random variable $d\ind{G_{\le k}}_{I_1}(X_i',X_j')$, and a union bound over the $\binom{n}{2}$ choices of $i$ and $j$ shows that some valid choice of $X_1',\dots,X_n'$ exist.

  Consider the set $X'\defeq X_1'\cup\cdots\cup X_n'$.
  Applying the $(\alpha',\gamma',\delta')$-sparse property of $G_{k+1}$ to the interval $I_1$ and set $X'$ with equitable partition $X_1'\cup\cdots\cup X_n'$, there exists an interval $I_2\subset I_1$ and indices $i$ and $i'$ and sets $W\subset X_i$ and $W'\subset X_{i'}$ such that 
  \begin{enumerate}
  \item[(i)] $|I_2|\ge \alpha'|I_1|\ge \alpha\alpha'|I|$, 
  \item[(ii)] $|W|=|W'|\ge \gamma'|X'|>\gamma'|X_1| \ge \gamma\gamma'|X|$, and 
  \item[(iii)] $d\ind{G_{k+1}}_{I_2}(W,W') < \delta'$.
  \end{enumerate}
  We thus have
  \begin{align}
    d_{I_2}\ind{G_{\le k+1}}(W,W') 
    \ &= \ d_{I_2}\ind{G_{\le k}}(W,W')  +  d_{I_2}\ind{G_{k+1}}(W,W') \nonumber\\
    \ &\le \ (\gamma')^{-2}d_{I_1}\ind{G_{\le k}}(X_i,X_{i'}) + d_{I_2}\ind{G_{k+1}}(W,W') \nonumber\\
    \ &< \ (\gamma')^{-2}n^2\delta + \delta'. 
  \label{}
  \end{align}
  In the first inequality, we used that $|W|=|W'|\ge \gamma'|X_i| = \gamma'|X_{i'}|$, and in the last inequality, we used item (iii) in each of the two lists above.
  This proves that $G_{\le k+1}$ is $(\alpha\alpha',\gamma\gamma',n^2(\gamma')^{-2}\delta + \delta', 2)$-sparse.
\end{proof}
We now have the following corollary.
\begin{lemma}
  Let $n\ge 4$ and $q\ge 2$ be positive integers, and let $G$ be a $q$-colored edge-ordered complete graph.
  Suppose that, for every $k=1,\dots,q$ and every $\delta_1\in(0,\frac{1}{n})$, the edge-ordered graph $G_k$ is $(n^{-2},\delta_1^n,\delta_1)$-sparse.
  Then, for all integers $k=1,\dots,q$ and $h\ge 1$, and for all $\delta_4\in(0,\frac{1}{n^2})$, we have $G_{\le k}$ is $(n^{-2^kn^{k-1}(2^h-1)},(\frac{\delta_4}{2^{2h}})^{hn(4n\log n)^{k-1}},\delta_4,2^h)$-sparse.
\label{lem:multi-2}
\end{lemma}
\begin{proof}
  For $k\ge 1$ and $h\ge 1$ and any $\delta_4\in (0,\frac{1}{n^2})$, let
  \begin{align}
    \alpha_{k,h} \ &\defeq \  n^{-2^kn^{k-1}(2^h-1)}, \qquad
    \gamma_{k,h} (\delta_4)
    \ \defeq \ \left(\frac{\delta_4}{2^{2h}}\right)^{hn(4n\log n)^{k-1}}.
  \label{}
  \end{align}
  We prove the lemma by induction on $k$ and $h$. For the base case $k=1$, we need to show, for all $h\ge 1$ and all $\delta_4\in(0,\frac{1}{n^2})$, that $G_1$ is $(n^{-2^{h+1}+2}, (\frac{\delta_4}{2^{2h}})^{hn},\delta_4, 2^h)$-sparse.
  The $h=1$ case is true by Lemma~\ref{lem:edge-1}, and the $h>1$ cases are true by Lemma~\ref{lem:edge-3}.

  For the induction step, we first show that if Lemma~\ref{lem:multi-2} is true for parameters $(k-1,\ceil{\log n})$, then it is true for parameters $(k,1)$.
  Then, we show that if Lemma~\ref{lem:multi-2} is true for parameters $(k,h)$ and $(k,1)$, then it is true for parameters $(k,h+1)$.

  For the first part of the induction step, fix $k\ge 2$ and $\delta_4\in(0,\frac{1}{n^2})$. 
  Let $h=\ceil{\log n}$.
  Assume that for $k'=1,\dots,q$  and $\delta\in(0,\frac{1}{n})$, the edge-ordered graph $G_k$ is $(n^{-2},\delta_1^n,\delta_1)$-sparse, and assume for induction that Lemma~\ref{lem:multi-2} is true for parameters $(k-1,h)$.
  Set 
  \begin{align}
    \delta_1   \ \defeq\ \frac{\delta_4}{2},\qquad
    \text{ and }\qquad
    \delta_4^* \ \defeq \  n^{-2}\delta_1^{2n+1}.
  \end{align}
  Then $G_{\le k-1}$ is $(\alpha_{k-1,h},\gamma_{k-1,h}(\delta_4^*),\delta_4^*,2^h)$-sparse by the induction hypothesis.
  Hence, by Fact~\ref{fact:sparse-2}, it is $(\alpha_{k-1,h},\gamma_{k-1,h}(\delta_4^*),\delta_4^*,n)$-sparse.
  Furthermore, $G_k$ is $(n^{-2}, \delta_1^n,\delta_1)$-sparse by our assumption in the statement of Lemma~\ref{lem:multi-2}.
  Hence, applying Lemma~\ref{lem:multi-1}, we have that $G_{\le k}$ is $(\alpha_{k-1,h}n^{-2}, \gamma_{k-1,h}(\delta_4^*)\cdot \delta_1^n, n^2\delta_1^{-2n}\delta_4^* + \delta_1, 2)$-sparse.
  We have
  \begin{align}
    \alpha_{k-1,h}n^{-2} \ &\ge \  n^{-2^{k-1}n^{k-2}(2^{\ceil{\log n}}-1)-2} \ > \  n^{-2^kn^{k-1}} \ = \ \alpha_{k,1} 
  \end{align}
  since $2^{\ceil{\log n}} \le 2n$.
  Additionally,
  \begin{align}
    \gamma_{k-1,h}(\delta_4^*)\cdot \delta_1^n
    \ &= \ \left( \frac{n^{-2}\delta_1^{2n+1}}{2^{2h}} \right)^{hn(4n\log n)^{k-2}}\cdot\delta_1^n \nonumber\\
    \ &> \ \left( \delta_1^{2n+4} \right)^{hn(4n\log n)^{k-2}}
    \ > \ \left( \frac{\delta_4}{4}\right)^{n(4n\log n)^{k-1}}
    \ = \ \gamma_{k,1}(\delta_4). 
  \end{align}
  In the first inequality, we used that $\delta_1^n > \delta_1^{hn(4n\log n)^{k-2}}$ and $\frac{n^{-2}}{2^{2h}} > \frac{n^{-2}}{4n^2} > \delta_1^2$.
  In the second inequality, we used that $\delta_1 > \frac{\delta_4}{4}$ and $(2n+4)\ceil{\log n} < 4n\log n$ for $n\ge 4$.
  Lastly, we have
  \begin{align}
    n^2\delta_1^{-2n}\delta_4^* + \delta_1 \ &= \ \delta_1+\delta_1 =  \delta_4.
  \label{}
  \end{align}
  Thus, by Fact~\ref{fact:sparse-2}, $G_{\le k}$ is $(\alpha_{k,1}, \gamma_{k,1}(\delta_4), \delta_4,2)$-sparse, completing the first induction step.

  For the second induction step, fix $k\ge 1$, $h\ge 1$, and $\delta_4\in(0,\frac{1}{n^2})$.
  Now assume that the $(k,h)$ and $(k,1)$ cases are true, and set $\delta_4^*=\frac{\delta_4}{2^{h+2}}$.
  By assumption, $G_{\le k}$ is $(\alpha_{k,h}, \gamma_{k,h}(\delta_4),\delta_4,2^h)$-sparse and $(\alpha_{k,1}, \gamma_{k,1}(\delta_4^*), \delta_4^*, 2)$-sparse, so, by Lemma~\ref{lem:edge-2} with $t=2^{h}$, we have $G_{\le k}$ is
  \begin{align}
    \left( \alpha_{k,h}^2\alpha_{k,1}, \half\gamma_{k,h}(\delta_4)\cdot\gamma_{k,1}(\delta_4^*), \delta_4, 2^{h+1} \right)\text{-sparse.}
  \label{}
  \end{align}
  We know $\alpha_{k,h}^2\alpha_{k,1}=\alpha_{k,h+1}$, and  
  \begin{align}
    \half\gamma_{k,h}(\delta_4)\cdot\gamma_{k,1}(\delta_4^*)
    \ &= \ \half \left( \frac{\delta_4}{2^{2h}} \right)^{hn(4n\log n)^{k-1}} \left(\frac{\delta_4}{2^{h+4}}\right)^{n(4n\log n)^{k-1}} \nonumber\\
    \ &< \ \left( \frac{\delta_4}{2^{2h+2}} \right)^{(h+1)n(4n\log n)^{k-1}} 
    \ =  \ \gamma_{k,h+1}(\delta_4). 
  \label{}
  \end{align}
  We conclude that $G_{\le k}$ is $(\alpha_{k,h+1}, \gamma_{k,h+1}(\delta_4), \delta_4,2^{h+1})$-sparse.
  This completes the induction, proving Lemma~\ref{lem:multi-2}.
\end{proof}

We now prove Theorem~\ref{thm:multi}.
\begin{proof}[Proof of Theorem~\ref{thm:multi}]
  If $n\le 3$, the edge-ordered Ramsey number is simply the usual Ramsey number, so assume $n\ge 4$.
  By Theorem~\ref{thm:edge-1}, Theorem~\ref{thm:multi} is true for $q=2$, so assume $q\ge3$.
  Let $N=2^{8^{q+1}n^{2q-2}(\log n)^q}$, let $\varepsilon=\exp(-2^{3q-3}n^{2q-2}(\log n)^q)$, and let $\delta_4=n^{-2^qn^{q-1}}$.
  Let $G$ be an $\varepsilon$-regular edge-ordered complete graph on $N$ vertices, which exists by Lemma~\ref{lem:edge-0}.
  Fix a $q$-coloring of $G$ and recall that, for $i=1,\dots,q$, $G_i$ is the edge-ordered graph on $N$ vertices consisting of the edges of $G$ with color $i$.
  Suppose for $i=1,\dots,q-1$, there exists an edge-ordered graph $H_i$ such that $G_i$ contains no copy of $H_i$.
  We now show that, for any edge-ordered graph $H$ on $n$ vertices, the edge-ordered graph $G_q$ contains a copy of $H$.
  Fix such an $H$.

  By Lemma~\ref{lem:edge-1}, for any $\delta_1\in(0,1/n)$, the graphs $G_1,\dots,G_{q-1}$ are each $(n^{-2},\delta_1^n,\delta_1)$-sparse.
  Then we may apply Lemma~\ref{lem:multi-2} to the edge-ordered graph $G$ with parameters $q-1$, $h=\ceil{\log n}$, and $\delta_4$ to obtain that $G_{\le q-1}$ is $(n^{-2^qn^{q-1}(2^h-1)}, (\frac{\delta_4}{2^{2h}})^{hn(4n\log n)^{q-1}}, \delta_4, 2^h)$-sparse.
  Apply this sparseness property to interval $I=[1,\binom{N}{2}]$ and vertex subset $X=[N]$.
  As $2^h-1\le 2n-3$ and $2^{q-1}n^{q-2}\ge 4n$, there exists an interval $I''$ of size at least $n^{-2^{q-1}n^{q-2}(2^h-1)}|I| \ge n^{-2^qn^{q-1}+12n}|I|$ and $2^h$ sets $W_1,\dots,W_{2^h}$ each with size satisfying 
  \begin{align}
    |W_i|
    \ \ge \  \left(\frac{\delta_4}{2^{2\ceil{\log n}}}\right)^{\ceil{\log n}n(4n\log n)^{q-2}}|X| 
    \ > \  \left(n^{-2^qn^{q-1}-3}\right)^{(1+\log n)n(4n\log n)^{q-2}}|X|  \nonumber\\
    \ > \  \left(n^{-2^{q}n^{q-1}}\right)^{2n\log n(4n\log n)^{q-2}}|X| 
    \ = \ \varepsilon|X| 
  \end{align}
  such that $d_{I''}\ind{G_{\le q-1}}(W_i,W_j) < \delta_4$ for all $1\le i<j\le 2^h$, and in particular all $1\le i < j \le n$.
  The second inequality uses that $2\ceil{\log n}<3\log n$ for $n\ge 4$.
  The third inequality uses that $(1 + \frac{3}{2^qn^{q-1}})(1+\log n) < 2\log n$ for $n\ge 4$ and $q\ge 3$.
  
  Let $J_{1,2},\dots,J_{n-1,n}$ be $\binom{n}{2}$ intervals forming an equitable partition of $I'$, ordered according to the order of the edges of $H$, as in \eqref{eq:equitable}.
  Let $\alpha=\frac{|J_{1,2}|}{\binom{N}{2}}$.
  In this way, $\alpha \ge \frac{|I''|}{n^2\binom{N}{2}} \ge n^{12n-2}\delta_4$.
  Let $F$ be an unordered $n$-partite graph with parts $W_1,\dots,W_n$, such that the edges between $W_i$ and $W_j$ are precisely the edges of $G$ between $W_i$ and $W_j$ whose label is in $J_{i,j}$.
  In this way, a clique in $F$ forms a copy of $H$ in $G$.
  Color an edge of $F$ red if its color in $G$ is one of $1,\dots,q-1$.
  For all $i$ and $j$ with $1\le i < j\le n$, we have $J_{i,j}\subset I''$.
  Hence, in $F$, by construction of $W_1,\dots,W_{n}$, between any distinct $W_i$ and $W_j$, the fraction of pairs that are red edges is at most $\delta_4$.
  Recall that $|W_i|> \varepsilon N$ for all $i$, so, as $G$ is $\frac{\varepsilon}{2}$-regular, for all $1\le i < j\le n$, the pair $(W_i,W_j)$ is $(\frac{|J_{i,j}|}{\binom{N}{2}},\frac{\varepsilon}{2})$-regular in the graph $G_{J_{i,j}}$, and hence in the graph $F$.
  By Fact~\ref{fact:reg-1}, every pair $(W_i,W_j)$ is $(\alpha,\varepsilon)$-regular in the graph $F$.

  As $\alpha \le \frac{1}{\binom{n}{2}} < \frac{1}{n}$, and $\delta_4<\frac{\alpha}{n^{12n-2}}<\frac{\alpha}{8n^2}$, and $\varepsilon < \frac{1}{8}\alpha^{2n+1}$, we may apply Lemma~\ref{lem:edge-4} with parameters $\alpha,\delta_3=\delta_4$, and $\varepsilon$ to the graph $F$ with parts $W_1,\cdots, W_n$ and the coloring defined above.
  This tells us that $F$ contains an $n$-clique with no red edges.
  By definition of the coloring of $F$, this clique forms a copy of $H$ in $G$ that is monochromatic in color $q$.
  This completes the proof.
\end{proof}

\section{Sparse graphs}
\label{sec:sparse}
In this section, we prove Theorem~\ref{thm:deg}, which states that the edge-ordered Ramsey number of an edge-ordered $d$-degenerate graph $H$ on $n$ vertices is at most $n^{600d\log(d+1)}$.
\begin{proof}[Proof of Theorem~\ref{thm:deg}]
  If $n\le 3$ vertices, the edge-ordered Ramsey number is the usual Ramsey number, so assume $n\ge 4$.
  Let $N=n^{600d^3\log (d+1)}$, let $\varepsilon=2n^{-100d^3\log (d+1)}$, let $\delta_1=n^{-5d}$, and let $\delta_2 = n^{-29d^2}$.
  Let $G$ be an $\varepsilon$-regular edge-ordered complete graph on $N$ vertices, which exists by Lemma~\ref{lem:edge-0}.
  Assume for contradiction that $G$ is colored such that the red subgraph $G_1$ and the blue subgraph $G_2$ each have no monochromatic copy of $H$.
  By Lemma~\ref{lem:edge-3} with $h=\ceil{\log(d+1)}$, the edge-ordered graph $G_1$ is $(\alpha_{h},\gamma_h,\delta_2,2^h)$-sparse, where 
  \begin{align}
    \alpha_h \ \defeq \ n^{-2^{h+1}+2}\ge n^{-4d+2},
    \qquad\text{ and }\qquad
    \gamma_h \ \defeq \   \left(\frac{\delta_2}{2^{h+2}(n-d)}\right)^{h(d+1)} > \delta_2^{3d\log(d+1)}.
  \end{align}
  In the second inequality, we used that $h=\ceil{\log(d+1)}$ so $h(d+1) \le 2d\log(d+1)$ and $2^{h+2}(n-d) \le 8d(n-d) < \delta_2^{-1/2}$. 
  Thus, $G_1$ is $(n^{-4d+2}, \delta_2^{3d\log(d+1)},\delta_2, d+1)$-sparse.
  Hence, there exists an interval $I$ of size at least $n^{-4d+2}\binom{N}{2}$ and subsets $W_1,\dots,W_{d+1}$ of size at least $\delta_2^{3d\log(d+1)}N$ such that, for all $j\neq j'$, we have $d_{I}\ind{G_1}(W_{j},W_{j'}) < \delta_2$.
  
  Since $H$ is $d$-degenerate, it is $(d+1)$-colorable.
  Fix a $(d+1)$-coloring of $H$ using colors $1,\dots,d+1$ and let $j_i$ denote the color of vertex $i$.
  There exist vertex subset $X_1,\dots,X_n$ of size $\ceil{n^{-1}\min_j|W_j|}$ such that all the $X_i$'s are disjoint and $X_i\subset W_{j_i}$.
  By choosing $X_1,\dots,X_n$ uniformly at random over all such ways of choosing such subsets, we also can guarantee that  $d_I\ind{G_1}(X_i,X_{i'}) < n^2\delta_2$ for all $(i,i')\in E(H)$:
  all such $i$ and $i'$ are of distinct colors, so the expectation of $d_I\ind{G_1}(X_i,X_{i'})$ is $d_I\ind{G_1}(W_{j_i},W_{j_{i'}})$, which is less than $\delta_2$, so the probability that $d_I\ind{G_1}(X_i,X_{i'}) \ge n^2\delta_2$ is less than $\frac{1}{n^2}$ by Markov's inequality, and the desired property follows from the union bound.

  Let $X=X_1\cup\cdots\cup X_n$.
  By Lemma~\ref{lem:edge-1}, $G_2$ is $(n^{-2}, \frac{1}{n-d}\delta_1^{d+1}, \delta_1)$-sparse with respect to $H_2$.
  Hence, there exists $I'\subset I$, an edge $(i,i')\in E(H_2)$ with $i\neq i'$, and vertex subsets $W_1'\subset X_i$ and $W_2'\subset X_{i'}$ such that
  \begin{enumerate}
  \item[(i)] $|I'|\ge n^{-2}|I| \ge n^{-4d}\binom{N}{2}$,
  \item[(ii)] $|W_1'|,|W_2'| > \frac{1}{n-d}\delta_1^{d+1}|X|
  \ge n^{-1}\delta_1^{2d}|X| \ge n^{-2}\delta_1^{2d}\delta_2^{3d\log(d+1)}N > \varepsilon N$.
  \item[(iii)] $d\ind{G_2}_{I'}(W_1',W_2') < \delta_1$.
  \end{enumerate}
  Above, we used that $n^{-1}\delta_1^{2d}|X_i| \varepsilon N$.
  Then we have
  \begin{align}
    d\ind{G_1}_{I'}(W_1',W_2') 
    \ &\le \ n^2\delta_1^{-4d}d\ind{G_1}_{I'}(X_i,X_{i'})
    \ \le \ n^2\delta_1^{-4d}d\ind{G_1}_{I}(X_i,X_{i'})
    \ < \ n^4\delta_1^{-4d}\delta_2. 
  \end{align}
  Hence, by $\varepsilon$-regularity of $G$, since $|I'| > \varepsilon\binom{N}{2}$ and $|W_1'|,|W_2'| > \varepsilon N$, we have
  \begin{align}
    n^{-4d} - \varepsilon 
    \ \le \  d\ind{G}_{I'}(W_1',W_2') 
    \ &= \  d\ind{G_1}_{I'}(W_1',W_2') + d\ind{G_2}_{I'}(W_1',W_2')  
    \ < \  \delta_1 + n^4\delta_1^{-4d}\delta_2
    \ \le \ 2n^{-5d}.
  \end{align}
  Since $d\ge 1$, $n\ge4$, and $\varepsilon < \frac{1}{2}n^{-4d}$, this is a contradiction and completes the proof.
\end{proof}

\section{Edge-labeled graphs}
\label{sec:rpt}

Given a $k$-partite graph $F = (V, E)$ with $k$-partition $V = V_1 \cup V_2 \cup \cdots\cup V_k$, a \emph{cylinder} $K$ is a set of the form $K=W_1\times \cdots\times  W_k$ where $W_i\subseteq V_i$ for $i=1,\dots,k$.
For a cylinder $K$, we write $V_i(K) = W_i$ for all $i = 1,2,\dots ,k$. 
A cylinder $K$ is $\varepsilon$-regular if all the $\binom{k}{2}$ pairs of subsets $(W_i, W_j)$ for $1 \le i < j \le k$ are $\varepsilon$-regular. A \emph{cylinder partition} $\mathcal{K}$ is a partition of the cylinder $V_1 \times V_2 \times  \cdots \times  V_k$ into cylinders. 
The partition $\mathcal{K}$ is \emph{cylinder-$\varepsilon$-regular} if all but at most an $\varepsilon$-fraction of the $k$-tuples $(v_1,\dots,v_k) \in V_1\times \cdots\times  V_k$ are in $\varepsilon$-regular cylinders in the partition $\mathcal{K}$. 
\begin{theorem}[\cite{DLR95}]
  Let $0 < \varepsilon < \frac{1}{2}$ and $\beta=\varepsilon^{k^2\varepsilon^{-5}}$.
  Suppose that $F=(V,E)$ is a $k$-partite graph with $k$-partition $V=V_1\cup\cdots\cup V_k$.
  Then there exists a cylinder-$\varepsilon$-regular partition of $\mathcal{K}$ of $V_1\times \cdots\times  V_k$ into at most $4^{k^2\varepsilon^{-5}}$ parts such that for each $K\in\mathcal{K}$ and $i\in[k]$, we have $|V_i(K)|\ge \beta|V_i|$.
\end{theorem}
The above regularity lemma can be extended to multiple colors in a standard way (see e.g.\ Section 1.9 of \cite{KS96}).
Recall that, for graphs $F$ colored by $1,\dots,q$, we let $F_s$ denote the subgraph obtained by keeping the edges of $F$ of color $s$.
\begin{theorem}
  Let $q,k\ge 2$ be integers, $\varepsilon\in(0,\frac{1}{2})$ and $\beta=\varepsilon^{k^2\varepsilon^{-5}}$.
  Suppose that $F=(V,E)$ is an $k$-partite graph with $k$-partition $V=V_1\cup\cdots\cup V_k$ whose edges are colored by $[q]$.
  Then there exists a partition $\mathcal{K}$ of $V_1\times \cdots\times V_k$ into at most $4^{k^2\varepsilon^{-5}}$ cylinders that is cylinder-$\varepsilon$-regular in each of the graphs $F_1,\dots,F_q$ and such that, for each $K\in\mathcal{K}$ and $i\in[k]$, we have $|V_i(K)|\ge \beta|V_i|$.
\label{thm:dlr-2}
\end{theorem}
It is worth noting the surprising fact that the bounds in the multicolor cylinder regularity lemma do not depend on the number $q$ of colors. The proof follows a similar density increment argument to \cite{DLR95}, and we sketch the proof below.
Recall $d\ind{F_s}(U,U') = \frac{e\ind{F_s}(U,U')}{|U_i||U_j|}$ where $e\ind{F_s}(U,U')$ is the number of edges between $U$ and $U'$ in graph $F_s$, i.e.\ of color $s$.
Consider the potential function 
\begin{align}
  q(\mathcal{K}) \ &= \  \sum_{U_1\times \cdots\times U_k\in\mathcal{K}}^{} \frac{|U_1|\cdots|U_k|}{|V_1|\cdots|V_k|} \sum_{1\le i < j\le k}^{} \sum_{s\in[q]}^{} d\ind{F_s}(U_i,U_j)^2. 
\end{align}
For all $1\le i < j\le n$, we have $\sum_{s\in[q]}^{} d\ind{F_s}(U_i,U_j)^2\le \left(\sum_{s\in[q]}^{} d\ind{F_s}(U_i,U_j)\right)^2 = 1$.
Hence, as $\mathcal{K}$ forms a partition of $V_1\times \cdots\times  V_k$, we have $0\le q(\mathcal{K})\le \binom{k}{2}$.
We start with the trivial partition, and obtain in steps a sequence of refinements for which the potential function increases until we arrive at a partition which is cylinder-$\varepsilon$-regular in each of $F_1,\dots,F_s$. From a partition $\mathcal{K}_i$ in step $i$, if it is not 
is cylinder-$\varepsilon$-regular in each of $F_1,\dots,F_s$, then there is some $F_j$ for which $\mathcal{K}_i$ is not cylinder-$\varepsilon$-regular. Using this, we obtain in step $i+1$ a refinement $\mathcal{K}_{i+1}$ of $\mathcal{K}_i$ such that each part of $\mathcal{K}_i$ is refined into at most $4$ parts (and hence $|\mathcal{K}_{i+1}| \leq 4 |\mathcal{K}_{i}|$), 
and for each $K\subset K'$ with $K\in\mathcal{K}_{i+1}$ and $K'\in \mathcal{K}_i$, we have $|V_i(K)|\ge \varepsilon|V_i(K')|$ for all $i\in[k]$, and $q(\mathcal{K}_{i+1})\ge q(\mathcal{K}_i)+\varepsilon^5$.
It follows that there are at most $\binom{k}{2}\varepsilon^{-5}$ steps before arriving at a cylinder-$\varepsilon$-regular partition $\mathcal{K}$. Hence, there are at most $4^{\binom{k}{2}\varepsilon^{-5}}$ parts and $|V_i(K)|\ge \varepsilon^{\binom{k}{2}\varepsilon^{-5}}|V_i|$ for all $K\in\mathcal{K}$ and $i\in[k]$.

\begin{corollary}
  \label{cor:dlr}
  Let $q,k\ge 2$ be an integers, $\varepsilon \in(0, \frac{1}{q})$, and $\beta=k^{-1}\varepsilon^{k^2\varepsilon^{-5}}$.
  Suppose that $F=(V,E)$ is an $N$-vertex graph colored in $q$ colors.
  There exist $k$ pairwise disjoint vertex subsets $U_1,\dots,U_k$ of size at least $\beta N$ such that, for all $1\le i < j\le k$, the pair $(U_i,U_j)$ is $\varepsilon$-regular in each of $F_1,\dots,F_q$.
\end{corollary}
\begin{proof}
Let $V=V_1\cup\cdots\cup V_k$ be an arbitrary equitable $k$-partition.
Apply Theorem~\ref{thm:dlr-2} to obtain a partition $\mathcal{K}$ of $V_1\times\cdots\times V_k$ that is cylinder $\varepsilon$-regular in each of $F_1,\dots,F_q$.
For each $s\in[q]$, all but at most an $\varepsilon$ fraction of $k$-tuples $(v_1,\dots,v_k)$ are in $\varepsilon$-regular cylinders.
Since $\varepsilon q < 1$, there exists a cylinder $U_1\times \cdots\times U_k$ that is $\varepsilon$-regular in each of $F_1,\dots,F_q$. Each $U_i$ has size at least $\varepsilon^{k^2\varepsilon^{-5}}|V_i|$, which is at least $k^{-1}\varepsilon^{k^2\varepsilon^{-5}}N$, as desired.
\end{proof}

\begin{proof}[Proof of Theorem~\ref{thm:rpt}]
The cases $n=1,2$ are trivial, so we may assume $n\ge 3$.
Let $L=q(m-1)+1$.
Let $Q=q\binom{L}{m}$.
Let $k$ be the Ramsey number $r(K_n; Q)$.
Let $\varepsilon = (\frac{1}{8qL})^n$.
Let $\beta=k^{-1}\varepsilon^{-k^2\varepsilon^{-5}}$.
Let $N=32L\beta^{-2}$.
In this way, $Q = 2^{\Theta(m\log q)}, k\le \binom{Qn}{n,\dots,n} \le 2^{O(Qn\log Q)} \le 2^{n2^{\Theta(m\log q)}}$, and $N\le 2^{2^{n2^{\Theta(m\log q)}}}$.
Let $L=\{1,\dots,L\}$.
For an edge-labeled graph $G$, an edge-label $\ell$, and vertex subsets $U$ and $U'$, let $d\ind{G}_\ell(U,U') = \frac{e\ind{G}_\ell(U,U')}{|U||U'|}$, where $e\ind{G}_\ell(U,U')$ is the number of edges between $U$ and $U'$ with label $\ell$.

We claim there exists an edge-labeled complete graph $G$ on $N$ vertices with labels in $[L]$ having the following property: for every label $\ell\in [L]$, and every pair of disjoint vertex sets $U$ and $W$ of size at least $\beta N$, we have $d\ind{G}_\ell(U,W)\ge \frac{1}{2L}$.
To see this, consider an edge-labeling $\varphi:E(G) \to[L]$ obtained by assigning an uniformly random label in $[L]$ independently to each edge.
For any $\ell\in [L]$, and $U,W\subset [N]$ of size at least $\beta N \ge \beta^{-1}$, the quantity $e_\ell(U,W)$ is a binomial random variable $B(|U||W|, \frac{1}{L})$, so by the Chernoff bound, the probability this is less than $\frac{1}{2L}|U||W|$ is at most $e^{-\frac{1}{8L}|U||W|}$. By the union bound over the at most $2^{2N}$ choices of $U,W$ with $|U|,|W|\ge \beta N$ and $L$ choices of $\ell\in [L]$, the probability the random edge-labeling fails to have the desired property is at most $$L2^{2N}e^{-\frac{1}{8L}(\beta N)^2}=L2^{2N}e^{-4N}<1,$$ so the desired edge-labeling exists.

Let $G$ be an edge-labeled complete graph as above, and consider a $q$-coloring $\chi:E(G)\to [q]$ of $G$.
From this $q$-coloring, obtain a $qL$-coloring $\chi^*$ of $G$ by coloring each edge $e$ by the pair $(\chi(e),\varphi(e))$ where $\varphi(e)$ is the label of $e$ in $G$.
For $s\in[q]$ and $\ell\in [L]$, let $G_s$ denote the edge-labeled subgraph of $G$ consisting of edges of color $s$ in the coloring $\chi$, and let $G_{(s,\ell)}$ denote the subgraph of $G$ consisting of the edges of color $(s,\ell)$ in the coloring $\chi^*$.
By Corollary~\ref{cor:dlr} with parameters $q'=qL$, $k$, $\varepsilon$, $\beta$, the graph $G$, and the coloring $\chi^*$, there exist $k$ pairwise disjoint vertex subsets $U_1,\dots,U_k$ of size at least $\beta N$ such that, for any $1\le i < j\le k$ and any color $(s,\ell)\in[q]\times [L]$, the pair $(U_i,U_j)$ is $\varepsilon$-regular in the graph $G_{(s,\ell)}$.

Consider a complete graph $F_0$ on vertices $1,\dots,k$ where edge $(i,j)$ is colored by the pair $(s,S)$, where $s\in[q]$ is a color and $S$ is a subset of $m$ labels $\ell$ such that $d_{\ell}\ind{G_s}(U_i,U_j)\ge \frac{1}{2L}$ for all $\ell\in S$. Such a pair always exists: for each label $\ell\in [L]$, there exists some color $s$ such that $d_{\ell}\ind{G_s}(U_i, U_j)\ge \frac{1}{q}d_\ell\ind{G}(U_i,U_j)\ge  \frac{1}{2qL}$, and, as there are $q(m-1)+1$ labels and $q$ colors, there exists some color $s$ such that at least $m$ labels $\ell$ satisfy $d_{\ell}\ind{G_s}(U_i, U_j)\ge \frac{1}{2qL}$.
As we chose $k=r(K_n;Q)$ with $Q=q\binom{L}{m}$, there exists $s\in[q]$ and $S \subset [L]$ of size $m$ such that $F_0$ has an $n$-clique monochromatic in the color $(s,S)$, which, without loss of generality, is formed by vertices $1,\dots,n$.
This means that, for all $1\le i < j \le n$ and all $\ell\in S$, we have $d_{\ell}\ind{G_s}(U_i,U_j)\ge \frac{1}{2qL}$.

We complete the proof of Theorem~\ref{thm:rpt} by showing that $G_s$ contains as a subgraph every possible edge-labeled clique $H$ on $n$ vertices with labels in $S$. For $1\le i < i' \le n$, let $\ell_{i,i'}\in S$ denote the label of edge $(i,i')$ in $H$.
Let $F$ be the $n$-partite graph on $U_1\cup\cdots\cup U_n$ where, for $1\le  i < i' \le n$, an edge is kept between $U_i$ and $U_{i'}$ if it has color $s$ and label $\ell_{i,i'}$. In this way, an $n$-clique in $F$ is a copy of $H$ in $G$  which is monochromatic in color $s$.

Let $p = \frac{1}{2qL}$.
For $1\le i < i' \le n$, let $p_{i,i'} = d\ind{F}(U_i,U_{i'})$, so that $p_{i,i'} = d\ind{G_s}_{\ell_{i,i'}}(U_i,U_{i'})\ge p$.
Since the pair $(U_i,U_{i'})$ is $\varepsilon$-regular in the graph $G_{(s,\ell_{i,i'})}$, it is $\varepsilon$-regular in the graph $F$, and thus $(p_{i,i'},\varepsilon)$-regular in the graph $F$.
Since $\varepsilon = (p/4)^n$, by Lemma~\ref{lem:count}, the number of $n$-cliques in $F$ is at least
\begin{align}
  \left( 1 - \frac{4\varepsilon n}{p^n} \right)\cdot \prod_{i=1}^{n} |W_i|\cdot \prod_{1\le i < i' \le n}^{} p_{i,i'}
  > 0.
\label{}
\end{align}
Thus, $F$ has an $n$-clique, and thus $G$ has a monochromatic copy of $H$.
\end{proof}

\section{Concluding remarks}
\label{sec:conclusion}

Many questions on Ramsey numbers can also be asked of edge-ordered Ramsey numbers.
It would be interesting to make progress on these questions.
In this section, we describe a few extensions of our results and some directions for future work.

\subsection{Closing the gap for complete graphs}
\label{sec:gap}
A major open question left by our work is to better estimate how large $r_{\edge}(H)$ can be for edge-ordered complete graphs on $n$ vertices.
We showed that $r_{\edge}(H)\le 2^{O(n^2\log^2n)}$ for all edge-ordered graphs $H$ on $n$ vertices.
We also know $r_{\edge}(H)\ge r(H)$.
Hence, when $H$ is an edge-ordered complete graph, $r_{\edge}(H)$ is always at least exponential in the number of vertices of $H$.
For complete graphs with a lexicographical edge-ordering, this lower bound is tight, but for general edge-ordered complete graphs $H$, we do not know whether this is tight.

In some settings, we are interested in guaranteeing a property that is stronger than having a monochromatic copy of a single graph.
For one such property, we have a slightly better lower bound.
\begin{definition}
  Given a family $\mathcal{F}$ of edge-ordered graphs, an edge ordered graph $G$ is \emph{Ramsey $\mathcal{F}$ universal} if, for every two-coloring of the edges of $G$, there exists a monochromatic subgraph containing a copy of every edge-ordered graph in $\mathcal{F}$.
\end{definition}
Let $\mathcal{F}_n$ be the family of all edge-ordered graphs on $n$ vertices.
Theorem~\ref{thm:edge-1} finds an edge-ordered graph on $2^{O(n^2\log^2n)}$ vertices that is Ramsey $\mathcal{F}_n$ universal, and Theorem~\ref{thm:multi} proves an analogous result for $q$ colors.
On the other hand, a Ramsey $\mathcal{F}_n$ universal graph $G$ needs $N\ge 2^{\Omega(n\log n)}$ vertices to simply fit all the graphs in $\mathcal{F}_n$ as subgraphs of $G$:  there are at most $N^n$ different $n$-vertex edge-ordered complete graphs that can appear in $G$, but there are $\binom{n}{2}!\ge2^{\Omega(n^2\log n)}$ edge-ordered complete graphs, so $N\ge 2^{\Omega(n\log n)}$.

\subsection{Edge-ordered Ramsey numbers of other graphs}

For edge-ordered graphs $H_1$ and $H_2$, let $r_\edge(H_1,H_2)$ denote the smallest $N$ such that there exists an edge-ordered graph $G$ on $N$ vertices such that every red-blue coloring of the edges of $G$ either contains a red copy of $H_1$ or a blue copy of $H_2$.

\subsubsection*{Sparse graphs}
We conjecture that Theorem~\ref{thm:deg} can be improved.
\begin{conjecture}
  \label{conj:deg-2}
  If $H$ is an edge-ordered $d$-degenerate graph on $n$ vertices, then $r_\edge(H) \le n^{O(d)}$. 
\end{conjecture}

The Burr-\Erdos conjecture has a natural analogue for edge-ordered Ramsey numbers that is stronger than Conjecture~\ref{conj:deg-2}, namely that, for all $d\ge 1$, there exists a constant $c(d)$ such that every edge-ordered $d$-degenerate graph $H$ on $n$ vertices satisfies $r_{\edge}(H)\le c(d)n$.
At the moment, we could not rule out this statement, as we do not know of edge-ordered $d$-degenerate graphs with superlinear edge-ordered Ramsey number.
However, we conjecture such examples exist, and in fact that Conjecture~\ref{conj:deg-2} is tight up to the constant in the exponent.
\begin{conjecture}
  \label{conj:deg}
  For every $d\ge 1$, there exist an infinite family of edge-ordered $d$-degenerate graphs $H$ such that $r_\edge(H) \ge n^{\Omega(d)}$, where $n$ is the number of vertices of $H$.
\end{conjecture}
As evidence towards Conjecture~\ref{conj:deg}, we note that, for $n$ sufficiently large, a random edge-ordered graph on $N \le n^{0.99d}$ vertices almost surely contain no copies of a fixed edge-ordered $d$-degenerate graph with the maximum number of edges.
A $d$-degenerate graph on $n\ge d+1$ vertices can have $m=dn-\binom{d}{2}$ edges, and let $H$ be such a graph with an arbitrary edge-ordering.
If $G$ is a randomly edge-ordered complete graph, the expected number of copies of $H$ is at most $N^n/m! < N^n(e/m)^m < N^n(3/dn)^{dn} \ll 1$.

Another natural question is to better estimate the edge-ordered Ramsey numbers for paths, trees, and cycles.
Balko and Vizer \cite{BalkoVslides} proved that the edge-ordered Ramsey number of a path with a monotone edge-ordering is linear in the number of vertices.
By Theorem~\ref{thm:deg}, the edge-ordered Ramsey number of paths, trees, and cycles is polynomial in the number of vertices, but the polynomials are of large degree.

\subsubsection*{Bounded chromatic number}

For the usual Ramsey number, all dense graphs, i.e.\ graphs on $n$ vertices with $\Omega(n^2)$ edges, have qualitatively the same behavior of the Ramsey number: by considering a random coloring, all dense graphs have a Ramsey number exponential in the number of vertices.
It is not clear whether the same is true for edge-ordered Ramsey numbers.

Balko and Vizer prove that, if $H_1$ and $H_2$ are edge-ordered graphs on $n$ vertices and $H_2$ is bipartite with $m$ edges, then $r_\edge(H_1,H_2) \le 2^{O(nm\log n)}$.
Our techniques improve this bound and generalize it to graphs of larger chromatic number.
Let $\mathcal{F}_{n,t}$ denote the family of all edge-ordered graphs on $n$ vertices with chromatic number at most $t$, and recall $\mathcal{F}_n$ is the family of edge-ordered graphs on $n$ vertices.
\begin{theorem}
  \label{thm:chr}
  For all positive integers $n$ and $t$ with $t\ge 2$, there exists an Ramsey $\mathcal{F}_{n,t}$ universal edge-ordered graph on $2^{O(nt(\log n)(\log t))}$ vertices.
  Additionally, $r_\edge(H_1,H_2)\le 2^{O(nt(\log n)(\log t))}$ for all $H_1\in \mathcal{F}_n$ and $H_2\in\mathcal{F}_{n,t}$.
\end{theorem}
For brevity, we sketch a proof of the second part here and omit the details.
Similar to the proofs of Theorems~\ref{thm:edge-1} and Theorem~\ref{thm:deg}, we take $G$ to be an $\varepsilon$-regular edge-ordered complete graph on $N$ vertices where $\varepsilon$ is roughly $N^{-1/6}$.
We show this $G$ gives both guarantees in Theorem~\ref{thm:chr}.
Suppose that $G$ is 2-colored with no red copy of some edge-ordered graph $H_1\in\mathcal{F}_n$.
Then $G$ is $(n^{-2},\delta_1^n,\delta_1)$-sparse for all sufficiently small $\delta_1$.
Fix some edge-ordered graph $H_2\in\mathcal{F}_{n,t}$.
Using Lemma~\ref{lem:edge-3} with parameter $h=\ceil{\log t}$, we find an interval $I$ and $t$ parts that have few red-$I$ edges between them in red graph.
As in the proof of Theorem~\ref{thm:deg}, from these $t$ parts, we obtain $n$ parts $W_1,\dots,W_n$ such that, for $(i,i')\in E(H_2)$, there are few red-$I$ edges between $W_i$ and $W_{i'}$.
The counting argument of Lemma~\ref{lem:edge-4} finds a clique in an appropriately chosen graph $F$, corresponding to a blue copy of $H_2$ in the graph $G$. 
Hence, we found a blue copy of $H_2$ for all $H_2\in\mathcal{F}_{n,t}$, completing the proof.

By Theorem~\ref{thm:chr}, we have tighter estimates on the edge-ordered Ramsey number of dense graphs of small chromatic number  than for complete graphs.
By Theorem~\ref{thm:chr}, if $t$ is a constant, then for all $H\in \mathcal{F}_{n,t}$ with $\Omega(n^2)$ edges, the edge-ordered Ramsey number is between $2^{\Omega(n)}$ and $2^{O(n\log n)}$.
Furthermore, similar to an argument in Subsection \ref{sec:gap}, any Ramsey $\mathcal{F}_{n,t}$ universal edge-ordered graph has at least $2^{\Omega(n\log n)}$ vertices, so the smallest Ramsey $\mathcal{F}_{n,t}$ universal graphs have $2^{\Theta(n\log n)}$ vertices.

\subsection{Relating vertex-ordered and edge-ordered Ramsey numbers}

Edge-ordered Ramsey numbers somewhat resemble \emph{(vertex-)ordered Ramsey numbers}.
A \emph{vertex-ordering} of a graph is a total ordering of its vertices.
Alternatively, a vertex-ordering can be given by an assignment of a unique integer label to each vertex in the graph.
We say two vertex-ordered graphs $H_1$ and $H_2$ are \emph{equivalent} if their is an isomorphism between $H_1$ and $H_2$ preserving the ordering of the vertices.
Given a vertex-ordered graph $H$, the \emph{vertex-ordered Ramsey number} $r_<(H)$ is the minimum $N$ such that any two-coloring of the edges of a complete graph on vertices $1,\dots,N$ contains a monochromatic subgraph equivalent to $H$.
Note that, when $H$ is a complete graph, we have $r_<(H) = r(H)$.
Conlon, Fox, Lee, and Sudakov \cite{ConlonFLS17} proved that $r_<(H)\le r(H)^{c\log^2 n}$ for general vertex-ordered graphs $H$ on $n$ vertices.

In some cases, vertex-ordered Ramsey numbers and edge-ordered Ramsey numbers can be related.
Let $H_<$ be a vertex-ordered graph, and $H_\edge$ be the graph $H_<$ with the corresponding lexicographical edge-ordering.
Then, $r_\edge(H_\edge)\le r_<(H_<)$.
To see this, take a complete graph $G$ on vertices $1,\dots,N=r_<(H_<)$ with a lexicographical edge-ordering.
In $G$, the edge-ordered copies of $H_\edge$ in $G$ are exactly the vertex-ordered copies of $H_<$, so $G\xrightarrow{\edge}H_{\edge}$.

In general, there doesn't appear to be a clear way to connect vertex-ordered and edge-ordered Ramsey numbers.
Indeed, in some instances, the bounds can be quite far apart.
For example, if $H$ is a matching on $n$ vertices, the vertex-ordered Ramsey number, for most vertex-orderings, is super-polynomial in $n$ \cite{ConlonFLS17,BCKK15}, but the edge-ordered Ramsey number is the usual Ramsey number, and hence linear in $n$.

\subsection{Edge-ordered Ramsey numbers of hypergraphs}
Hypergraph Ramsey numbers have also been actively studied (see, e.g.\ \cite{ConlonFS15,MS19}).
Their existence was initially proved by Ramsey~\cite{Ramsey30}, and better bounds were obtained by \Erdos and Rado \cite{ER52}.
We can define \emph{edge-ordered hypergraph} as a hypergraph with an ordering on the hyperedges.
We say two edge-ordered hypergraphs $\mathcal{H}_1$ and $\mathcal{H}_2$ are \emph{equivalent} if there is an isomorphism between them preserving the ordering of the hyperedges.
We write $\mathcal{G}\xrightarrow{\edge}\mathcal{H}$ if every 2-coloring of $\mathcal{G}$ contains a monochromatic subhypergraph equivalent to $\mathcal{H}$.
In this way, for any $k\ge 2$, one can define the edge-ordered Ramsey number of a $k$-uniform edge-ordered hypergraph $\mathcal{H}$ as the smallest $N$ such that there exists a $k$-uniform edge-ordered hypergraph $\mathcal{G}$ on $N$ vertices such that $\mathcal{G}\xrightarrow{\edge}\mathcal{H}$.
A general result of Hubi\v cka and Ne\v set\v ril (Theorem 4.33 in \cite{HuNe}) proves the existence of hypergraph edge-ordered Ramsey numbers, but the bounds are enormous.
A natural problem is to give a better estimate hypergraph edge-ordered Ramsey numbers.

\subsection{Edge-induced Ramsey numbers}
A graph $H$ is said to be an induced subgraph of $G$ if $V(H) \subset V(G)$ and two vertices of $H$ are adjacent if and only if they are adjacent in $G$.
For graphs $G$ and $H$, we write $G\xrightarrow{\inuce} H$ if any two-coloring of the edges of $G$ contains an monochromatic \emph{induced} copy of $H$.
The existence of induced Ramsey numbers was proven independently by Deuber~\cite{Deuber75},  Erd\H{o}s, Hajnal, and Pach \cite{EHP00}, and \Rodl \cite{Rodl73}.
The best known bounds were given by Conlon, Fox, and Sudakov \cite{CFS12}.

We can generalize induced Ramsey numbers to edge-ordered graphs.
For edge-ordered graphs $G$ and $H$, we write $G\xrightarrow{\edgeind}H$ if any two-coloring of the edges of $G$ contains a monochromatic \emph{induced} copy of $H$, i.e.\ a monochromatic induced subgraph equivalent to $H$.
For an edge-ordered (not necessarily complete) graph $H$, we say the \emph{edge-induced Ramsey number} $r_{\edgeind}(H)$ is the smallest $N$ such that there exists an edge-ordered graph $G$ on $N$ vertices such that $G\xrightarrow{\edgeind}H$.
Our arguments give the same kind of bound as Theorem~\ref{thm:edge-1} on edge-induced Ramsey numbers.
\begin{theorem}
  For any edge-ordered graph $H$, we have $r_{\edgeind}(H)\le 2^{O(n^2\log^2n)}$.
\label{}
\end{theorem}
For brevity, we sketch a proof here and omit the details.
We follow the same argument as in Theorem~\ref{thm:edge-1}, except that now we define $G$ by taking a random graph chosen from $G_{N,1/2}$, as opposed to a complete graph.
We then order the edges of $G$ randomly as before.
The quasi-randomness of the non-edges allows us to find induced copies of the smaller graph $H$ as easily as before:
using the quasi-randomness of the non-edges, we can prove, similar to Lemma~\ref{lem:edge-1}, that if the red subgraph $G_1$ has no induced copy of $H$, it is $(n^{-2},\delta_1^{O(n)},\delta_1,2)$-sparse for all sufficiently small $\delta_1$.
We can then iterate this as in Lemma~\ref{lem:edge-3} to show that $G_1$ is $(n^{-4n+2},\delta_2^{O(n\log n)},\delta_2,n)$-sparse for all sufficiently small $\delta_1$.
Then we apply Lemma~\ref{lem:edge-4} to find a clique with no red edges on appropriately chosen graph $F$, corresponding to a blue induced copy of $H$ in $G$.

\subsection{Size edge-ordered Ramsey numbers}

The size Ramsey number of a graph $H$, introduced by Erd\H{o}s, Faudree, Rousseau and Schelp \cite{EFRS78}, denoted $\hat r(H)$, is the minimum number $\hat r(H)$ such that there exists a graph $G$ with $\hat r(H)$ edges such that any two-coloring of the edges of $G$ has a monochromatic copy of $H$.
We can similarly define the \emph{size edge-ordered Ramsey number} of an edge-ordered graph $H$ to be the minimum number $\hat r_{\edge}(H)$ such that there exists an edge-ordered graph $G$ with $\hat r_{\edge}(H)$ edges such that any two coloring of the edges of $G$ has a monochromatic edge-ordered copy of $H$.
By taking $G$ to be a complete graph, we have that $\hat r(H)\le \binom{r(H)}{2}$ for all graphs $H$, and similarly, for edge-ordered graphs, we also have $\hat r_\edge (H)\le \binom{r_{\edge}(H)}{2}$.
For the usual Ramsey number, \Chvatal observed (see \cite{EFRS78}) that equality holds: $\hat r(K_n) = \binom{r(K_n)}{2}$.
A similar question can be asked for size edge-ordered Ramsey numbers.
\begin{question}
  Is $\hat r_\edge(H) = \binom{r_\edge(H)}{2}$ for all edge-ordered complete graphs $H$?
\end{question}
A classic result of Beck \cite{B83} states that the size Ramsey number of paths is linear in the number of vertices.
This was later extended to show that the size Ramsey number is linear for bounded degree trees \cite{FriedmanP87} and cycles \cite{HaxellKL95}.
A similar question can be asked for size edge-ordered Ramsey numbers.
\begin{question}
  What families of edge-ordered graphs, have linear or near-linear size edge-ordered Ramsey number?
\end{question}
The edge-ordered Ramsey number is linear for stars and matchings, as all edge-ordered stars are equivalent and all edge-ordered matchings are equivalent.
For monotone edge-ordered paths, Balko and Vizer \cite{BalkoVslides} showed that the edge-ordered Ramsey number is linear in $n$, so the size edge-ordered Ramsey number is at most quadratic.

\subsection{Other Ramsey type problems for edge-ordered graphs}
Ramsey type questions have previously been asked for edge-ordered graphs.
For some special edge-ordered graphs, it is possible to show that every edge-ordered graph has a copy of that graph.
This is the case for monotone paths.
\Chvatal and \Komlos \cite{ChvatalK71} asked for the largest integer $f(N)$, such that, for any ordering of the edges of $K_N$, there exists a simple path of length $f(N)$ whose edges form a monotone sequence.
Graham and Kleitman \cite{GrahamK73} proved that $f(N)\ge \sqrt{N} - O(1)$.
Despite much attention over the years, substantial progress has only been made recently.
Milans \cite{Milans17} proved $f(N)\ge \Omega((\frac{N}{\log N})^{2/3})$.
This was improved by Buci\'c, Kwan, Pokrovskiy, Sudakov, Tran, and Wagner \cite{BucicKPSTZ18} to $f(N)\ge \Omega(N^{1-o(1)})$.
No $o(N)$ upper bound is known for $f(N)$.
The best upper bound is given by Calderbank, Chung, and Sturtevant \cite{CCS84}, who proved $f(N)\le \frac{N}{2} + o(N)$.

A similar question has also been asked for random edge-orderings.
Lavrov and Loh \cite{LavrovL14} showed that, in a random edge-ordering, there is a length $N-1$ path with probability at least about $1/e$ and a length $0.85N$ path with probability $1-o(1)$.
Martinsson~\cite{Martinsson19} recently improved the former result, showing there is a length $N-1$ path with probability $1-o(1)$.

\section{Acknowledgements}

We thank Gabor Tardos for introducing us to this topic at Building Bridges II, conference celebrating the 70th birthday of L\'aszl\'o Lov\'asz. We would also like to thank Martin Balko and M\'at\'e Vizer for sharing a preprint and helpful comments.


\providecommand{\bysame}{\leavevmode\hbox to3em{\hrulefill}\thinspace}
\providecommand{\MR}{\relax\ifhmode\unskip\space\fi MR }
\providecommand{\MRhref}[2]{%
  \href{http://www.ams.org/mathscinet-getitem?mr=#1}{#2}
}
\providecommand{\href}[2]{#2}

\end{document}